\documentclass[final,leqno,onefignum,onetabnum]{siamltex1213}
\title{STABLE CELL-CENTERED FINITE VOLUME DISCRETIZATION FOR BIOT
EQUATIONS}

\author{Jan Martin Nordbotten\footnotemark[2]}
\slugger{sinum}{xxxx}{xx}{x}{x--x}

\usepackage{amsmath,amssymb}
\usepackage{booktabs}
\allowdisplaybreaks

\def\abs#1{\lvert #1\rvert}
\def\norm#1{\lVert #1\rVert}
\def\dblbr#1{[[#1]]}
\def\aver#1{\langle #1\rangle}

\newcommand*{\xib}{\boldsymbol{\xi}}
\newcommand*{\nub}{\nu}
\newcommand*{\pib}{\boldsymbol{\pi}}

\def\overkern#1#2#3{
 {} \mskip#1mu \overline{\mskip-#1mu #3 \mskip-#2mu} \mskip#2mu {}
}

\newcommand*{\bnabla}{\overkern11\nabla}

\newcommand*{\C}{\mathbb{C}}
\newcommand*{\R}{\mathbb{R}}
\newcommand*{\LBB}{\mathit{LBB}}
\newcommand*{\FV}{\mathit{FV}}
\newcommand*{\ch}{\mathit{ch}}

\newcommand*{\eb}{\boldsymbol{e}}
\newcommand*{\fb}{\boldsymbol{f}}
\newcommand*{\gb}{\boldsymbol{g}}
\newcommand*{\kb}{\boldsymbol{k}}
\newcommand*{\nb}{\boldsymbol{n}}
\newcommand*{\qb}{\boldsymbol{q}}
\newcommand*{\ub}{\boldsymbol{u}}
\newcommand*{\vb}{\boldsymbol{v}}
\newcommand*{\wb}{\boldsymbol{w}}
\newcommand*{\xb}{\boldsymbol{x}}
\newcommand*{\yb}{\boldsymbol{y}}
\newcommand*{\zb}{\boldsymbol{z}}

\newcommand*{\Hb}{\boldsymbol{H}}
\newcommand*{\Ib}{\boldsymbol{I}}
\newcommand*{\Ob}{\boldsymbol{0}}

\newcommand*{\Tb}{\boldsymbol{T}}

\newcommand*{\Cc}{\mathcal{C}}
\newcommand*{\Dc}{\mathcal{D}}
\newcommand*{\Fc}{\mathcal{F}}
\newcommand*{\Gc}{\mathcal{G}}
\newcommand*{\Hc}{\mathcal{H}}
\newcommand*{\Tc}{\mathcal{T}}
\newcommand*{\Vc}{\mathcal{V}}

\newcommand*{\Hbc}{\boldsymbol{\mathcal{H}}}

\newcommand*{\afr}{\boldsymbol{\mathfrak{a}}}
\newcommand*{\bfr}{\boldsymbol{\mathfrak{b}}}

\newcommand*{\Afr}{\mathfrak{A}}
\newcommand*{\Bfr}{\mathfrak{B}}

\newtheorem{comment}[theorem]{\normalfont\textit{Comment}}
\newtheorem{remark}[theorem]{Remark}
\newtheorem{statement}[theorem]{Statement}
\newtheorem{theo}{Theorem}
\newtheorem{condition}{Condition}

\begin{document}
\maketitle

\renewcommand{\thefootnote}{\fnsymbol{footnote}}

\footnotetext[2]{Department of Mathematics,
University of Bergen and
Department of Civil and Environmental Engineering,
Princeton University.}

\renewcommand{\thefootnote}{\arabic{footnote}}

\begin{abstract}
In this paper we discuss a new discretization for the Biot equations. The
discretization treats the coupled system of deformation and flow
directly, as opposed to combining discretizations for the two
separate sub-problems. The coupled discretization has the following
key properties, the combination of which is novel: 1) The variables
for the pressure and displacement are co-located, and are as sparse as
possible (e.g.\ one displacement vector and one scalar pressure per
cell center). 2) With locally computable restrictions on grid types,
the discretization is stable with respect to the limits of incompressible
fluid and small time-steps. 3) No artificial stabilization term has been
introduced. Furthermore, due to the finite volume structure embedded in
the discretization, explicit local expressions for both momentum-balancing
forces as well as mass-conservative fluid fluxes are available.

We prove stability of the proposed method with respect to all relevant
limits. Together with consistency, this proves convergence of the
method. Finally, we give numerical examples verifying both the analysis
and convergence of the method.
\end{abstract}

\begin{keywords}\end{keywords}

\begin{AMS}\end{AMS}

\pagestyle{myheadings}
\thispagestyle{plain}
\markboth{JAN MARTIN NORDBOTTEN}{CELL-CENTERED FINITE VOLUME
DISCRETIZATION FOR BIOT EQUATIONS}

\section{Introduction}\label{sec1}

Deformable porous media are becoming increasingly important in
applications. In particular, the emergence of strongly engineered
geological systems such as CO$_2$ storage \cite{nor11,rut10}, geothermal
energy \cite{tes06}, and shale-gas extraction all require analysis
of the coupling of fluid flow and deformation. Beyond the subsurface,
physiological processes are increasingly simulated exploiting the Biot
models~\cite{bad09}.

With this motivation, we consider the following model problem poroelastic
media~\cite{bio41}
\begin{subequations}\label{eq1.1}
\begin{alignat}{2}
&\nabla\cdot\pib=\fb_{\ub} &&\text{in }\Omega\label{eq1.1a}\\
&\pib=\C:\nabla\ub-\alpha p\Ib&&\text{in } \Omega\label{eq1.1b}\\
&\alpha\nabla\cdot\ub + \rho p+\tau \nabla\cdot\qb=f_p&\quad& \text{in } \Omega\label{eq1.1c}\\
&\qb=-\kb\nabla p&&\text{in } \Omega\label{eq1.1d}\\
&\ub=\gb_{\ub,D}&&\text{on } \Gamma_{\ub,D}\label{eq1.1e}\\
&\pib\cdot\nb=\gb_{\ub,N}&&\text{on } \Gamma_{\ub,N}\label{eq1.1f}\\
&p=g_{p,D}&&\text{on } \Gamma_{p,D}\label{eq1.1g}\\
&\qb\cdot\nb=g_{p,N}&&\text{on } \Gamma_{p,N}\label{eq1.1h}
\end{alignat}
\end{subequations}

Equations \eqref{eq1.1} arise from an implicit (backward Euler) time
discretization of the linear Biot's equations, and in this context
$\tau$ represents the time-step. The domain $\Omega$ is a bounded
connected polygonal subset of $\R^d$, with boundary $\partial\Omega
=\Gamma_D\cup\Gamma_N$. The notation adopted in equation~\eqref{eq1.1}
utilizes the variables vector displacement ($\ub$), scalar pressure
($p$), tensor Biot stress ($\pib$) and vector Darcy flux ($\qb$). The
spatially dependent material stiffness tensor is denoted $\C$, the Biot
coupling coefficient ($\alpha$), the fluid compressibility ($\rho$),
and the permeability ($\kb$). We denote source-terms by $f$ and boundary
conditions as $g$.

All parameters are in general positive (definite) functions of space,
and the discretization formulated herein considers arbitrary spatial
variability of the parameter functions. Nevertheless, we will for
the sake of simplicity and conciseness of presentation suppress this
spatial dependence in the current section and in the convergence
analysis given in Section~\ref{sec5}. Thus when parameters are treated
as constants, it should be understood in the sense as e.g.\ $\rho
=\inf_{\xb\in\Omega}\rho(\xb)$. For boundary conditions, we assume
that both $\Gamma_{\ub,D}$ and $\Gamma_{u,D}$ have positive measure: The
case of $\Gamma_{\ub,N}=\partial\Omega$ or $\Gamma_{p,N}=\partial\Omega$
can be accommodated, but the additional complications will not be
of sufficient interest to merit the additional notation and loss of
clarity in the presentation. Without loss of generality, we will by
subtracting any smooth function satisfying the boundary conditions
and correspondingly modifying the right-hand side, assign all boundary
conditions as homogeneous.

Under the stated conditions on parameters and boundary conditions, for
reasonable forcing functions equations \eqref{eq1.1} have a unique weak
solution $(\ub,p)\in(H^1(\Omega))^d\times H^1(\Omega)$ (throughout the
paper, we will implicitly consider the restriction of the function spaces
such that the Dirichlet boundary conditions are satisfied). Furthermore,
in the case where $\tau=0$, or equivalently, if the medium is impermeable
($\kb=0$), the compressible Stokes equations are recovered. This implies
that the regularity of the pressure is reduced, and we have a unique
weak solution $(\ub,p)\in(H^1(\Omega))^d\times L^2(\Omega)$. Finally,
if the incompressible case is considered, $\rho =0$, the pressure in
the Stokes equations is only defined up to a constant. This motivates
the following definition:

\begin{definition}\label{def1.1}
We denote a discretization of equations \eqref{eq1.1} as \emph{robust}
if there exists an estimate of the form
\[
\norm{\ub}_1+\tau \norm{p}_1+\rho \norm{p}_0+\abs{p}_0\le
C(\norm{\fb_{\ub}}+\norm{f_g})
\]
Where the constant $C$ is bounded independent of the
small parameters $\rho$ and $\tau$, and we denote by
$\abs{p}_0=\inf_{p_0}\norm{p-p_0}_0$.
\end{definition}

Our interest will be in obtaining a simple discretization which is robust,
which retains an explicit conservation principle for both momentum balance
\eqref{eq1.1a} and mass balance \eqref{eq1.1c} and which has a co-located cell-centered
data structure for both pressure and deformation. As our primary interest
lies in geological porous media, we will furthermore make no assumption
on the smoothness of the data beyond bounded measure.

The main difficulty associated with obtaining a robust discretization of
equations \eqref{eq1.1} arises from the three saddle-point systems
embodied in the formulation. In absence of the fluid pressure,
equations \eqref{eq1.1a} and \eqref{eq1.1b} represent a saddle-point system for the
solid sub-problem, and similarly, in absence of solid deformation
equations \eqref{eq1.1c} and \eqref{eq1.1d} represent a saddle-point system for the
fluid sub-problem. Finally, if mechanical stress and fluid fluxes are
eliminated, the displacement-pressure system is itself a saddle-point
system.

Disregard conservation properties, low order standard finite
elements are nevertheless not stable for the limit of incompressible
solid \cite{nag74,bre91}. Due to the saddle-point nature of the
displacement-pressure system, the displacement and pressure spaces must
furthermore be chosen as a Stokes-stable pair in order to satisfy a
Ladyzhenskaya--Babuska--Brezzi (LBB, or inf-sup) condition as $\rho,\tau
\to 0$ \cite{bre91}. For these reasons, stabilized formulations have
been considered to allow for arbitrary choices of finite element spaces
\cite{bad09}. While their stabilization can be motivated through various
formalisms, consistency is often lost. Similar problems appear for
co-located finite difference discretizations. While less literature is
available, a Brezzi--Pitkaranta type \cite{bre84} stabilization term has
also been analyzed in the context finite difference methods for the Biot
equations \cite{gas08}.

A natural option in order to obtain stable discretizations with retain
the explicit conservation structure is to construct a multi-field set of
inf-sup stable mixed finite element spaces. The fluid sub-problem can be
successfully discretized by mixed finite elements \cite{bre91}, and much
supporting literature exists for this problem. In contrast, the solid
sub-problem does not lend itself trivially to mixed finite elements, and
although much recent literature relates to this problem, simple low-order
element combinations appear to be impossible to devise \cite{arn02}. A
thorough numerical investigation of various combinations of mixed methods
for the full Biot equations has been reported \cite{hag12}. An alternative
to mixed finite elements (but similar in spirit) is to use staggered
grids for the mechanics and flow equations \cite{whe14,gas06}. The use of
staggered grids resolves stability issues, but leads to more restrictions
on admissible grids, and increases the technical difficulties associated
with constructing efficient solvers.

A novel cell-centered finite volume (FV) method for solid displacement
has recently been introduced as multi-point stress approximation (MPSA)
method by the author, for the purpose of consistent treatment of the
fluid and displacement equations \cite{nor14a}. This class of finite
volume methods, along with its counterpart for fluid flow, multi-point
flux approximation (MPFA) methods \cite{aav02}, is applicable to a wide
range of grids of relevance to industrial applications. In contrast to
control-volume finite element methods \cite{jas00}, these methods are
furthermore particularly well suited for problems with large material
contrasts \cite{eig05}, as are frequently seen in subsurface flows. The
MPSA and MPFA methods have recently been combined to solve Biot's
equations \cite{nor14b}. While that paper shows applications including
fractured media with complex geometries, the resulting discretization
is not robust in the limit of small timesteps $\tau$, and no theoretical
analysis is provided for the coupled scheme.

Theoretical properties for MPFA methods has long been
established by exploiting links to mixed finite element methods
\cite{kla06,whe12}. However, due to the lack of simple finite element
methods for elasticity and Biot's equations, of relevance for the
current work is the analysis within the hybrid finite volume framework
\cite{eym06}, and the recent analysis of the MPFA \cite{age10} and MPSA
methods \cite{nor15}.

Herein, we will expand on the finite volume framework as developed
in \cite{nor15} to establish a coupled discretization for the Biot
equations directly. This approach leads us to a new discretization,
which we refer to as MPFA/MPSA-FV. The discretization has the following
properties: A) The discretization of the fluid and mechanical sub-problem
are identical to the decoupled finite volume discretizations. B) When
local variables are eliminated by static condensation (as is typical for
finite volume methods), the resulting system of equations is in terms of
cell-center displacement and pressure only. C) We show that the co-located
discretization is naturally stable in the sense of Definition~\ref{def1.1}
without addition of any artificial stabilization term or stabilization
parameter. D) The discretization is consistent with the variational
finite volume formulation of the Biot system.

We note at this moment that Definition~\ref{def1.1} is not sufficient
to study the property of numerical locking for nearly incompressible
mechanical sub-problem. In order to analyze locking, a term needs
to be introduced accounting for the divergence of displacement,
$\norm{\nabla\cdot\ub}_0$. We choose not to include this in our analysis
for four reasons: a) Fully incompressible materials are of little interest
in the setting of Biot equations, since the coupling term is exactly
the compression. b) The mechanical discretization obtained herein is in
itself locking free on most grids \cite{nor15}. c) For the Biot equations,
locking-free schemes can be achieved at no additional complexity by
introducing a solid pressure in the formulation \cite{lew98}, and d)
The analysis becomes much more complex, both in practical terms and
in the sense of notation. Indeed, to illustrate points a) and d)
we note that for incompressible materials, the left-hand side of
equation \eqref{eq1.1c} is zero in the limit of $(\rho,\tau)\to 0$, such that
no regularity of pressure is expected in this limit. This issue can be
resolved by observing that within the context of a time discretization,
$f_b=O(\max(\rho,\tau,\nabla\cdot \ub))$, however this indicates some
of the technical issues arising in the analysis.

We structure the rest of the manuscript as follows: In the next
section we will recall a suitable discrete functional framework. In
Section~\ref{sec3} we will define our method within the framework of
variational finite volume methods. In Section~\ref{sec4}, we identify
the cell-centered discretizations through local static condensation. The
main result is stated in Section~\ref{sec5}, where we show stability
and convergence of the coupled discretization. In Section~\ref{sec6},
we provide numerical results which underscore the robustness claims in
all relevant parameter regimes, and provides numerical convergence rates
for smooth problems. Finally, Section~\ref{sec7} concludes the paper.

\section{Discrete functional framework}\label{sec2}

In this section we give the definition of our finite volume mesh
and discrete variables. The setting will be identical to the discrete
framework used for pure elasticity in \cite{nor15}. This mesh description
and discrete framework expands on those of \cite{age10}, which again
generalizes \cite{eym06}. As this section contains no novel material,
the presentation will be as brief as possible, and the reader is referred
to the references for further details.

\subsection{Finite volume mesh}\label{sec2.1}

We denote a finite volume mesh by the triplet $\Dc=(\Tc,\Fc,\Vc)$,
representing the mesh Tessellation, Faces, and Vertexes, such that:
\begin{itemize}
\item $\Tc$ is a non-overlapping partition of the domain
$\Omega$. Furthermore, let $m_K$ denote the $d$-dimensional measure of
$K\in \Tc$.
\item $\Fc$ is a set of faces of the partitioning $\Tc$. We consider only
cases where elements $\sigma\in\Fc$ are subsets of $d-1$ dimensional
hyper-planes of $\R^d$, and all elements $\sigma \in\Fc$ we associate
the $d-1$ dimensional measure $m_\sigma$. Naturally, the faces must
be compatible with the mesh, such that for all $K\in\Tc$ there exists
a subset $\Fc_K\subset \Fc$ such that $\partial K=\bigcup_{\sigma
\in\Fc_K}\sigma$.
\item $\Vc$ is a set of vertexes of the partitioning $\Tc$. Thus for
any $d$ faces $\sigma_i\in\Fc$, either their intersection is empty or
$\bigcap_i\sigma_i =s\in\Vc$.
\end{itemize}

Note that in the above (and throughout the manuscript), we abuse notation
by referring to the object and the index by the same notation. E.g., we
will by $K\in\Tc$ allow $K$ to denote the index, as in $\Fc_K$, but also
the actual subdomain of $\Omega$, such that the expressions $\partial K$
and $m_K=\int_K d\xb$ are meaningful.

Additionally, we state the following useful subsets of the mesh triplet,
which allows us to efficiently sum over neighboring cells, faces or
vertices:
\begin{itemize}
\item For each cell $K\in\Tc$, in addition to the faces $\Fc_K$, we
denote the vertexes of $K$ by $\Vc_K$. We will associate with each vertex
$s\in\Vc_K$ a subcell of $K$, identified by $(K,s)$, with a volume $m_K^s$
such that $\sum_{s\in\Vc_K}m_K^s =m_K$.
\item For each face $\sigma \in\Fc$, we denote the neighboring
cells $\Tc_\sigma$ and its vertex for $\Vc_\sigma$. Note that for all
internal faces $\Tc_\sigma$ will contain exactly two elements, while it
contains a single element when $\sigma\subset\partial\Omega$. We will
associate with each corner $s\in\Vc_\sigma$ a subface of $\sigma$,
identified by $(s,\sigma)$, with an area $m_\sigma^s$ such that
$\sum_{s\in\Vc_\sigma}m_\sigma^s =m_\sigma$.
\item For each vertex $s\in\Vc$, we denote the adjacent cells by $\Tc_s$
and the adjacent faces by $\Fc_s$.
\end{itemize}

We associate for each element $K\in\Tc$ a unique point (cell
center) $\xb_K\in K$ such that $K$ is star-shaped with respect to
$\xb_K$, and we denote the diameter of $K$ by $d_K$. Furthermore,
we denote the distance between cell centers $\xb_K$ and $\xb_L$
as $d_{K,L}=\abs{\xb_K-\xb_L}$. The grid diameter is denoted
$h=\max_{K\in\Tc}d_K$.

We associate with each face $\sigma$ its outward normal vector with
respect to the cell $K\in\Tc_\sigma$ as $\nb_{K,\sigma}$, and the
Euclidian distance to the cell center $d_K^\sigma$. For each subface
$(s,\sigma)$ we denote the subface center as $\xb_\sigma^s$ and a set of
quadrature points on $(s,\sigma)$ as $\Gc_\sigma^s$. For each quadrature
point $\beta\in\Gc_\sigma^s$ we associate the position $\xb_\beta$ and
weight $\omega_\beta$. In general, we will choose sufficient quadrature
points for exact integration of second-order polynomials, however for
simplex grids it is advantageous to choose only a single quadrature point.

The above definition covers all 2D grids of interest, and quite general 3D
grids. However, the definition disallows 3D grids with curved surfaces
(e.g.\ distorted cubes). A more general formulation appears to be
practicable \cite{age10}, but would come at the expense of additional
notation and analysis.

Regularity assumptions on the discretization $\Dc$ are detailed elsewhere
(see e.g.\ \cite{eym06}), we will in the interest of simplicity of
exposition henceforth assume that the classical grid regularity parameters
(grid skewness, internal cell angles, and coordination number of vertexes)
do not deteriorate.

\subsection{Discrete variables and norms}\label{sec2.2}

We detail the three discrete spaces used in our analysis below. They
represent, respectively, the space of cell variables $\Hc_\Tc$, cell and
discontinuous face variables $\Hc_\Dc$, and finally cell and continuous
face variables $\Hc_\Cc$.

The following discrete space is classical \cite{eym06}, and is the space
where the final cell-centered discretization will take its values.

\begin{definition}\label{def2.1}
For the mesh $\Tc$, let $\Hc_\Tc(\Omega)\subset L^2 (\Omega)$ be the
set of piece-wise constant functions on the cells of the mesh $\Tc$.
\end{definition}

As with the dual interpretation of the elements $K\in\Tc$, the space
$\Hc_\Tc (\Omega)$ is isomorphic to the space of discrete variables
associated with the cell-center points $\xb_K$. There should also be
no cause for confusion in the following when we also work with the
vector-valued spaces, then denoted by bold $\Hbc_\Tc$.

For the space $\Hc_\Tc$ we introduce the inner product
\[
[u,v]_\Tc=\sum_{K\in\Tc}\sum_{\sigma\in\Fc_K}\frac{m_\sigma}{d_{K,\sigma}}(\gamma_\sigma
u-u_K) (\gamma_\sigma v-v_K)
\]
and its induced norm
\[
\norm{u}_\Tc=([u,u]_\Tc)^{1/2}
\]
Here the operator $\gamma_\sigma u$ interpolates the piecewise constant
values of $\Hc_\Tc$ onto the faces of the mesh, weighted by the distances
$d_{K,\sigma}$.
\[
\gamma_\sigma
u=\Bigl(\sum_{K\in\Tc_\sigma}\frac{u_K}{d_{K,\sigma}}\Bigr)/\Bigl(\sum_{K\in\Tc_\sigma}d_{K,\sigma}^{-1}\Bigr)
\quad\text{for all }\sigma\in\Fc;\ \sigma\notin\Gamma_D
\]
For Dirichlet boundary edges, $\sigma \in\Gamma_D$, we take $\gamma_\sigma
u=0$. Equivalently, the operator $\gamma_\sigma u$ can be defined as the
value which minimizes the definition of the norm $\norm{u}_\Tc$. The
norm defined above is naturally identified as an $H^1$ type norm for
the space. We will furthermore need the discrete $L^2$ inner product
for $\Hc_\Tc$
\[
[u,v]_{\Tc,0}=\sum_{K\in\Tc}m_K u_K v_K
\]
and its induced norm
\[
\norm{u}_{\Tc,0}=([u,u]_{\Tc,0})^{1/2}
\]

\begin{definition}\label{def2.2}
For the mesh triplet $\Dc$, let $\Hc_\Dc$ be the set of real
scalars $(u_K,u_{K,s}^{\sigma,\beta})$, for all $K\in\Tc$, for all
$(s,\sigma)\in\Vc_K\times \Fc_K$ and for all $\beta\in\Gc_\sigma^s$.
\end{definition}

The space $\Hc_\Dc$ thus contains one unknown per cell, in addition to
multiple unknowns on each interior sub-face. This space was introduced
in order to control the space of rigid body motions when discretizing
the mechanical sub-problem \cite{nor14a}. As above, we will immediately
take $u_{K,s}^{\sigma,\beta}=0$ for all $\sigma \in\Gamma_D$.

We denote for all internal subfaces
$\dblbr{u}_s^{\sigma,\beta}=u_{R,s}^{\sigma,\beta}-u_{L,s}^{\sigma,\beta}$
for $u\in\Hc_\Dc$ and $\Tc_\sigma =\{R,L\}$ as the
jump in the discrete function $u$ across that edge. We
will also need a notion of an average face value, and
we denote similarly for all internal subfaces $\aver{u}_s^\sigma
=\sum_{\beta\in\Gc_s^\sigma}\omega_\beta\frac{u_{R,s}^{\sigma,\beta}+u_{L,s}^{\sigma,\beta}}2$.
For boundary edges $\sigma \in\partial\Omega$
only one function value is available and we define
$\dblbr{u}_s^{\sigma,\beta}=0$ and $\aver{u}_s^\sigma
=\frac1{m_s^\sigma}\sum_{\beta\in\Gc_s^\sigma}\omega_\beta
u_{R,s}^{\sigma,\beta}$. We now associate with the space $\Hc_\Dc$
the inner product
\begin{align*}
[u,v]_\Dc=&\sum_{K\in\Tc}\sum_{s\in\Vc_K}\sum_{\sigma
\in\Fc_s\cap\Fc_K}\frac{m_K^s}{d_{K,\sigma}^2}
(u_K-\aver{u}_s^\sigma)(v_K-\aver{v}_s^\sigma)\\
&+\frac{m_K^s}{d_{K,\sigma}^2}\frac1{m_s^\sigma}
\sum_{\beta\in\Gc_s^\sigma}\omega_\beta \dblbr{u}_s^{\sigma,\beta}
\dblbr{v}_s^{\sigma,\beta}
\end{align*}
and the induced norm
\[
\norm{u}_\Dc=[u,u]_\Dc^{1/2}
\]

\begin{definition}\label{def2.3}
For the mesh triplet $\Dc$, let $\Hc_\Cc$ be the set of real
scalars $(u_K,u_s^\sigma)$, for all $K\in\Tc$ and for all
$(s,\sigma)\in\Vc_K\times \Fc_K$.
\end{definition}

This last space is essential for obtaining the finite volume structure
of the scheme~\cite{age10}.{\emergencystretch5pt\par}

By introducing the natural interpolation operator $\Pi_\Dc\colon\Hc_\Cc\to
\Hc_\Dc$ as $(\Pi_\Dc u)_K=\nobreak u_K$; $(\Pi_\Dc u)_{K,s}^{\sigma,\beta}=u_s^\sigma$
for all $K\in\Tc$ and for all $(s,\sigma)\in\Vc_K\times \Fc_K$, we can
immediately define the inner product
\[
[u,v]_\Cc=[\Pi_\Dc u,\Pi_\Dc v]_\Dc
\]
and the induced norm
\[
\norm{u}_\Cc=[u,u]_\Cc^{1/2}
\]

In addition to the interpolation operator defined above, we shall need
a few more operators to move between function spaces.
\begin{itemize}
\item Let the operator $\Pi_\Tc\colon\Hc_\Dc\to \Hc_\Tc$ be defined as
$(\Pi_\Tc u)(x)=u_K$ for all $x\in K$ and $K\in\Tc$. Furthermore, as there
should be no reason for confusion we also define $\Pi_\Tc\colon\Hc_\Cc\to
\Hc_\Tc$ with as $(\Pi_\Tc u)(x)=(\Pi_\Tc \Pi_\Dc u)(x)=u_K$ for all
$x\in K$ and $K\in\Tc$. Finally, we also write $\Pi_\Tc\colon C(\Omega)\to
\Hc_\Tc$ as $(\Pi_\Tc u)(x)=u(x_K)$ for all $x\in K$ and $K\in\Tc$.
\item Let the operator $\Pi_\Cc\colon\Hc_\Dc\to \Hc_\Cc$ be defined as
$(\Pi_\Cc u)_K=u_K$; $(\Pi_\Cc u)_s^\sigma =\aver{u}_s^\sigma$ for all
$K\in\Tc$ and for all $(s,\sigma)\in\Vc_K\times \Fc_K$.
\end{itemize}

The spaces defined above satisfy the following inequalities.
\begin{itemize}
\item Discrete Poincar\'e inequality \cite{eym06}: For all $u\in\Hc_\Tc$
\[
\norm{u}_{\Tc,0}\le C_P \norm{u}_\Tc
\]
\item Inverse inequality \cite{eym06}: For all $u\in\Hc_\Tc$
\[
\norm{u}_{\Tc,0}\ge\sqrt{d}h\norm{u}_\Tc
\]
\item Relationship between $\Hc_\Tc$ and $\Hc_\Cc$ \cite{age10}: For
all $u\in\Hc_\Cc$
\[
\norm{\Pi_\Tc u}_\Tc\le \sqrt{d} \norm{u}_\Cc
\]
\item Relationship between $\Hc_\Cc$ and $\Hc_\Dc$ (trivial from
definitions): For all $u\in\Hc_\Dc$
\[
\norm{\Pi_\Cc u}_\Cc\le \norm{u}_\Dc
\]
\end{itemize}

Finally, we introduce local spaces $\Hc_{\Dc,s}\subset \Hc_\Dc$
for each $s\in\Vc$ defined such that $u\in\Hc_{\Dc,s}$ if
$u_{K,t}^{\sigma,\beta}=0$ for all $t\in\Vc$ with $t\neq s$ and $u_K=0$
if $s\notin\Vc_K$. Similarly, $\Hc_{\Tc,s}$ and $\Hc_{\Cc,s}$ are defined
through the operators defined above. The local spaces have the natural
(semi-)norms, which to be precise are given for all $u\in\Hc_\Dc$ as
\[
\norm{u}_{\Dc,s}^2=\sum_{K\in\Tc_s}\sum_{\sigma
\in\Fc_s\cap\Fc_K}\frac{m_K^s}{d_{K,\sigma}^2}
(u_K-\aver{u}_s^\sigma)^2+\frac{m_K^s}{d_{K,\sigma}^2} \frac1{m_s^\sigma}
\sum_{\beta\in\Gc_s^\sigma}\omega_\beta (\dblbr{u}_s^{\sigma,\beta})^2
\]
And for all $u\in\Hc_\Tc$ as
\[
\norm{u}_{\Tc,s}^2=\sum_{K\in\Tc_s}\sum_{\sigma
\in\Fc_s\cap\Fc_K}\frac{m_\sigma^s}{d_{K,\sigma}}(\gamma_\sigma u-u_K)^2
\]
Such that both
\[
\norm{u}_\Dc^2=\sum_{s\in\Vc}\norm{u}_{\Dc,s}^2 \quad\text{and}\quad
\norm{u}_\Tc^2=\sum_{s\in\Vc}\norm{u}_{\Tc,s}^2
\]

\section{Discrete mixed variational FV discretizations for
Biot}\label{sec3}

In this section, we will utilize the spaces defined in Section~\ref{sec2}
to establish a cell-centered finite volume method for Biot's
equations. The approach builds on the discrete mixed variational
formulation of existing methods for the pressure (cf.\ \cite{age10}
and \cite{nor15}) and displacement equations \cite{nor15}. The key
novel aspects arise due to the coupling terms arising in equations
\eqref{eq1.1}. The careful treatment of these terms will lead to a
naturally stable discretization, which improves on the method presented
in \cite{nor14b}. These issues will be crucial in the analysis, and
further highlighted in the numerical examples.

Our approach will exploit a discrete variational formalism for the
Biot problem, thus we introduce the variational form of equations
\eqref{eq1.1}: Find $(\ub,p)\in\Hb^1\times H^1$ such that
\begin{subequations}\label{eq3.1}
\begin{alignat}{2}
&(\C:\nabla\ub,\nabla\vb)-(\alpha p,\nabla\cdot\vb)=-(\fb_{\ub},\vb)&\quad&
\text{for all }\vb\in\Hb^1\label{eq3.1a}\\
&{-}(\nabla\cdot\ub,\alpha r)-(\rho p,r)-\tau (k\nabla p,\nabla r)=-(f_p,r)&\quad&
\text{for all }r\in H^1\label{eq3.1b}
\end{alignat}
\end{subequations}
While this form of the equation hides the stress and flux expressions,
it will be implied throughout (and made explicit in Section~\ref{sec3.2})
that the methods considered herein allow for explicit and local extraction
of mass conservative fluid fluxes, as well as momentum conserving surface
tractions. It is also important to note that the integration by parts
is essential for the coupling term in equation~\eqref{eq3.1a}, in order for
the system to still be well defined in the limit $\tau \to 0$, when
regularity of $p$ is reduced.

\subsection{Method formulation}\label{sec3.1}

Will formulate our discretizations on a discrete variational form,
based on equations \eqref{eq1.1}. To obtain both a finite volume
structure as well as a consistent method, we follow previous work
\cite{age08,nor14a,age10}, and introduce two notions of a discrete
gradient. First, we note that the discrete divergence provides by duality
a definition of a gradient, which thus exactly preserves the finite
volume structure of the governing equations for each cell $K\in\Tc$.

\begin{definition}\label{def3.1}
For each $K\in\Tc$ and each $s\in\Vc_K$ we define the \emph{finite volume
gradient} for all $\ub\in\Hbc_\Cc$:
\begin{equation}
(\widetilde{\nabla}\ub)_K^s=\frac1{m_K^s} \sum_{\sigma
\in\Fc_K\cap\Fc_s}m_\sigma^s (\aver{\ub}_{K,s}^\sigma
-\ub_K)\otimes\nb_{K,\sigma}
\label{eq3.2}
\end{equation}
\end{definition}

For the definition to make sense, we need to specify the averaging
notation in the natural way, that is to say $\aver{\ub}_{K,s}^\sigma
=\frac1{m_s^\sigma} \sum_{\beta\in\Gc_s^\sigma}\omega_\beta
u_{K,s}^{\sigma,\beta}$. We comment that in previous work, the finite
volume gradient has been referred to as the ``convergent gradient''
\cite{age10}. The finite volume gradient does not enjoy strong convergence
properties, and thus a notion of a gradient which is exact for locally
multi-linear discrete functions is needed.

\begin{definition}\label{def3.2}
For each $K\in\Tc$ and each $s\in\Vc_K$ we define the \emph{consistent
gradient} for all $\ub\in\Hbc_\Dc$:
\begin{equation}
(\bnabla\ub)_K^s=\sum_{\sigma \in\Fc_K\cap\Fc_s}(\aver{\ub}_{K,s}^\sigma
-\ub_K)\otimes\gb_{K,\sigma}^s
\label{eq3.3}
\end{equation}
\end{definition}

In order to satisfy the desired consistency property that
$(\bnabla\ub)_K^s$ is exact for linear displacements, the vectors
$\gb_{K,\sigma}^s$ must satisfy the system of equations:
\begin{equation}
\Ib_2=(\bnabla\xb)_K^s=\sum_{\sigma
\in\Fc_K\cap\Fc_s}(\aver{\xb}_{K,s}^\sigma -\xb_K)\otimes\gb_{K,\sigma}^s
\label{eq3.4}
\end{equation}
Here $\Ib_2$ is the $d$-dimensional second-order identity tensor. The
vectors $\gb_{K,\sigma}^s$ are unique for the grids considered herein,
but may be non-unique for more general 3D grids \cite{age10}.

For both discrete gradients, the corresponding divergence is obtained
way by either taking the trace of the gradient, or equivalently by
replacing the outer product by a dot product. Note that the consistent
and finite volume gradients are defined on the two distinct discrete
spaces $\Hbc_\Dc$ and $\Hbc_\Cc$, respectively.

A consistent finite volume formulation is now obtained by using the
finite volume gradient for the test functions and the consistent
gradient otherwise. With this in mind, we introduce the following
bilinear forms. For all $(\ub,\vb)\in\Hbc_\Dc\times \Hbc_\Cc$ and
$(p,r)\in\Hc_\Dc\times \Hc_\Cc$
\begin{align}
a_\Dc
(\ub,\vb)&=\sum_{K\in\Tc}\sum_{s\in\Vc_K}m_K^s\bigl(\C_K:(\bnabla\ub)_K^s:(\widetilde{\nabla}\vb)_K^s\bigr)
\label{eq3.5}\\
c_\Dc (p,r)&=\sum_{K\in\Tc}\sum_{s\in\Vc_K}m_K^s \bigl(k_K (\bnabla p)_K^s\cdot
(\widetilde{\nabla}r)_K^s\bigr)
\label{eq3.6}
\end{align}
With the coupling terms defined for all
\begin{align}
b_{\Dc,1} (\ub,r)&=-\sum_{K\in\Tc}\alpha_K (\Pi_\Tc r)_K
\sum_{s\in\Vc_K}m_K^s (\bnabla \cdot \ub)_K^s
\label{eq3.7}\\
b_{\Dc,2}^T (p,\vb)&=-\sum_{K\in\Tc}\alpha_K (\Pi_\Tc p)_K
\sum_{s\in\Vc_K}m_K^s (\widetilde{\nabla}\cdot \vb)_K^s
\label{eq3.8}
\end{align}

Note that the coupling terms are not in general transposes, due to the
different discrete differential operators. Thus the transpose notation
is used to indicate that the terms have a familiar structure.

To close the system, we need to associate the degrees of freedom in
$\Hc_\Dc$ with an appropriate polynomial order, and ensure a suitable
degree of continuity \cite{nor15}. Following the general paradigm for
MPFA and MPSA methods \cite{aav02,nor14a}, we require the solution to
be locally linear on each subcell associated with $(K,s)\in\Tc\times
\Vc_K$, and consistent with the gradient $\bnabla$. This is expressed
through a bilinear form defined for all $((\ub,p),\wb)\in(\Hbc_\Dc\times
\Hc_\Dc)\times (\Hbc_\Dc\times \Hc_\Dc)$ such that
\begin{multline}
d_\Dc
((\ub,p),\wb)=\sum_{K\in\Tc}\sum_{s\in\Vc_K}\sum_{\sigma\in\Fc_s}\sum_{\beta\in\Gc_\sigma^s}
\bigl((\ub,p)_{K,s}^{\sigma,\beta}-(\ub,p)_K-(\bnabla (\ub,p))_K^s\cdot
(\xb_\beta-\xb_K)\bigr)\\[-1ex]
\hskip5cm\cdot \bigl(\wb_{K,s}^{\sigma,\beta}-\wb_K-(\bnabla\wb)_K^s\cdot
(\xb_\beta-\xb_K)\bigr)
\label{eq3.9}
\end{multline}

Finally, weak continuity is enforced by minimizing jumps at quadrature
points for both the pressure and displacement, yielding the symmetric
bilinear form
\begin{equation}
e_\Dc \bigl((\ub,p),\wb\bigr)=\sum_{s\in\Vc}\sum_{\sigma
\in\Fc_s}\frac{\nub_s^\sigma}{m_s^\sigma}
\sum_{\beta\in\Gc_\sigma^s}\omega_\beta\dblbr{(\ub,p)}_s^{\sigma,\beta}\cdot\dblbr{\wb}_s^{\sigma,\beta}
\label{eq3.10}
\end{equation}

The positive weights $\nub_s^\sigma$ can in principle be chosen
arbitrarily, however a weighted harmonic mean of the constitutive
functions of the nearby cells $\Tc_\sigma$ seems beneficial in practice
\cite{nor14a}. The full discrete problem is a constrained miminization
problem, whose solution satisfies the following mixed variational system:
Find $((\ub_\Dc,p_\Dc),(\yb_\Cc,y_\Cc),\yb_\Dc)\in(\Hbc_\Dc\times \Hc_\Dc)\times
(\Hbc_\Cc\times \Hc_\Cc)\times (\Hbc_\Dc\times \Hc_\Dc)$ such that the
physical constraints
{\emergencystretch6pt
\begin{alignat}{2}
&a_\Dc (\ub_\Dc,\vb)+b_{\Dc,2}^T (p_\Dc,\vb)=-\int_\Omega \fb_{\ub}\cdot
\Pi_\Tc\vb\,d\xb&\quad&
\text{for all }\vb\in\Hbc_\Cc\label{eq3.11}\\
&
b_{\Dc,1} (\ub_\Dc,r)-[\rho_\Tc \Pi_\Tc p_\Dc,\Pi_\Tc r]_{\Tc,0}
\label{eq3.12}\\
&
\phantom{b_{\Dc,1} (\ub_\Dc,r)}
-\tau c_\Dc
(p_\Dc,r)=
-\int_\Omega f_p\cdot \Pi_\Tc r \,d\xb
&\quad&\text{for all }r\in\Hc_\Cc
\notag
\end{alignat}
}%
And the piece-wise linear approximation hold
\begin{equation}
d_\Dc \bigl((\ub_\Dc,p_\Dc),\wb\bigr)=0\quad\text{for all }\wb\in(\Hbc_\Dc\times
\Hc_\Dc)
\label{eq3.13}
\end{equation}
Subject to the solution constrained minimization
\begin{multline}\label{eq3.14}
\begin{aligned}
e_\Dc \bigl((\ub_\Dc,p_\Dc),(\wb,w)\bigr)&+a_\Dc (\wb,\yb_\Cc)+b_{\Dc,1} (\wb,y_\Cc)
+b_{\Dc,2}^T (w,\yb_\Cc)\\ 
&-[\rho_\Tc \Pi_\Tc w_\Dc,\Pi_\Tc y_\Cc]_{\Tc,0}
-\tau c(w,y_\Cc)+d_\Dc \bigl((\wb,w),\yb_\Dc\bigr)=0
\end{aligned}\\
\text{for all }(\wb,w)\in(\Hbc_\Dc\times \Hc_\Dc)
\end{multline}

In equation~\eqref{eq3.12} and \eqref{eq3.14}, we have introduced the
shorthand $\rho_\Tc=\Pi_\Tc \rho$. The components $\yb_\Cc$ and $\yb_\Dc$
are simply Lagrange multipliers for the constrained minimization problem,
and will not be of further interest. After local static condensation
(as made explicit below), they will not appear in the final cell-centered
method.

Equations \eqref{eq3.11}--\eqref{eq3.14} allows us to define the discrete
mixed variational formulation considered herein.

\begin{definition}\label{def3.3}
We refer to equations \eqref{eq3.11}--\eqref{eq3.14} as the
\emph{discrete mixed variational finite volume formulation for Biot's
equations}.
\end{definition}

{\makeatletter
\def\@begintheorem#1#2{\par\bgroup{{\itshape #1}\ #2. }\ignorespaces}
\makeatother
\begin{comment}\label{com3.4}
We will throughout the manuscript discuss the general setting where
$\bnabla\neq\widetilde{\nabla}$. However, it is well known that in the
special case of simplex grids, a symmetric method can be obtained by using
a single quadrature point (see e.g.\ \cite{eig05,age10,nor14a}). All
the results for the general setting applies to this case, with some
simplifications which will be noted when relevant.
\end{comment}
}

\subsection{Finite volume structure}\label{sec3.2}

In absence of the coupling terms (e.g.\ with $\alpha =0$ in the
bilinear form $b_\Dc$), the discrete mixed variational finite volume
formulation as given in equations \eqref{eq3.11}--\eqref{eq3.14}
is identical to the discretizations for elasticity and flow analyzed
previously \cite{nor15}. Furthermore, if the set of Gauss quadrature
points $\Gc_s^\sigma$ is reduced to a single point per subinterface
$(s,\sigma)$, the discretization for the flow problem reduces to the
well-known MPFA finite volume scheme \cite{age10,aav02}.

The finite volume structure arises due to the definition of the finite
volume gradients $\widetilde{\nabla}$. Indeed, we see that (with
reference to a canonical basis for $\Hc_\Cc$), the degrees of freedom
associated with cell centers imply that e.g.\ equation~\eqref{eq3.12}
can be equivalently re-written as:
\begin{multline}
\tau \sum_{\sigma \in\Fc_K}m_\sigma q_K^\sigma (p_\Dc)=\int_Kf_p\,d\xb
-\sum_{s\in\Vc_K}[m_K^s \alpha_K (\bnabla \cdot\ub)_K^s-\rho_K m_K p_K]
\\
\text{for all } K\in\Tc;
\label{eq3.15}
\end{multline}
Where the normal fluxes $q_K^\sigma$ are defined as
\begin{equation}
q_K^\sigma (p_\Dc)=-\biggl[\frac1{m_\sigma} \sum_{s\in\Vc_K}m_\sigma^s \kb_K
(\bnabla p)_K^s\biggr]\cdot \nb_{K,\sigma}
\label{eq3.16}
\end{equation}
Furthermore, the degrees of freedom associated with subface
variables in $\Hc_\Cc$, imply that for all internal faces $\sigma$,
equation~\eqref{eq3.12} reduces to (for $\{K,L\}=\Tc_\sigma$)
\begin{equation}
q_K^\sigma (p_\Dc)=-q_L^\sigma (p_\Dc)
\label{eq3.17}
\end{equation}
Equations \eqref{eq3.15} and \eqref{eq3.17} represents the conservation
structure of the scheme, while equation~\eqref{eq3.16} provides an
explicit local expression for the fluid normal flux. Together these
features identify the scheme as a finite volume method.

The finite volume structure arises in the same sense for the mechanical
sub-problem. First we introduce the notion of traction $\Tb(\nb)$,
which is the force on an internal surface with normal vector $\nb$,
and is related to the stress as $\Tb(\nb)=\pib\cdot\nb$. In
the continuous setting, it is the balance of tractions which
leads to equation~\eqref{eq1.1} \cite{tem00}. Now we note that
equation~\eqref{eq3.11} can be rewritten (with $\vb\in\Hbc_\Tc$) as
\begin{equation}
\sum_{\sigma \in\Fc_K}m_\sigma \Tb_K^\sigma (\ub_\Dc,p_\Dc;\nb_{K,\sigma})
=\int_K\fb_{\ub} \,d\xb
\quad\text{for all } K\in\Tc;
\label{eq3.18}
\end{equation}
Here again we note the explicit definition of the traction $\Tb_K^\sigma$
as
\begin{equation}
\Tb_K^\sigma (\ub_\Dc,p_\Dc;\nb_{K,\sigma})
=\biggl[\frac1{m_\sigma}\sum_{s\in\Vc_K}m_\sigma^s
\bigl(\C_K:(\bnabla\ub)_K^s-\alpha_K p_K\Ib_2\bigr)\biggr]\cdot\nb_{K,\sigma}
\label{eq3.19}
\end{equation}
And finally also the continuity of tractions reduces to (again for
$\{K,L\}=\Tc_\sigma$)
\begin{equation}
\Tb_K^\sigma (\ub_\Dc,p_\Dc;\nb_{K,\sigma})=-\Tb_L^\sigma
(\ub_\Dc,p_\Dc;\nb_{L,\sigma})
\label{eq3.20}
\end{equation}
It is important to note that both the momentum balance
(equation~\eqref{eq3.18}) and the continuity (equation~\eqref{eq3.20})
are written directly in terms tractions derived from the Biot stress. This
guarantees the correct notion of force balance in the method, and will
be reflected in the structure of the methods as seen in the subsequent
sections. Furthermore, the fact that the normal flux only depends on
pressure in equation~\eqref{eq3.16}, while the tractions depends on
both displacement and pressure in equation~\eqref{eq3.19} is consistent
with the governing physics of equation~\eqref{eq1.1}, and presages the
discrete structure which is revealed in equation~\eqref{eq4.7}.

The finite volume structure revealed by equations
\eqref{eq3.15}--\eqref{eq3.20} motivates the nomenclature of
Definition~\ref{def3.3}.

\section{Local problems and cell-centered discretization}\label{sec4}

To obtain a method of practical applicability, and in particular
a cell-centered scheme, it is necessary to be able to perform
a local static condensation within the framework of equations
\eqref{eq3.11}--\eqref{eq3.14}. For mixed finite-element formulations,
this is frequently achieved by numerical quadrature (see e.g.\
\cite{rus83,whe06,kla06}), however, in our setting no further
approximations are required. Indeed, we will see that equations
\eqref{eq3.11}--\eqref{eq3.14} by construction allow for local static
condensation to be performed. Such constructions are classical for finite
volume methods for elliptic problems (see e.g.\ \cite{eym06,aav96}).

As noted, the discrete mixed variational problem
\eqref{eq3.11}--\eqref{eq3.14} is a direct generalization of
the discretizations for elasticity and fluid pressure analyzed in
\cite{nor15}. Importantly, this implies that the local problems, and
hence many components of the coupled discretization, are identical to
their decoupled counterparts. This will greatly simplify the subsequent
analysis.

The structure of this section is as follows. We will first clearly
identify the local problems, and show that the static condensation
required to reach a cell-centered formulation is well-posed. This also
induces an interpolation from the space $\Hbc_\Tc\times \Hc_\Tc$ into
$\Hbc_\Dc\times \Hc_\Dc$. This interpolation forms the key to analyzing
the method, and allows us to express the cell-centered discretization.

\subsection{Local problems}\label{sec4.1}

The variational multiscale methods (VMS) \cite{hug98} as applied to
mixed problems \cite{arb04,nor09}, provides a suitable framework to
formalize the localization and static condensation of the system
problem \eqref{eq3.11}--\eqref{eq3.14}. We note that this use of the
VMS framework is different from the approach taken when deriving
additional stabilization terms \cite{bad09}. By identifying the
cell-center values $(\Hbc_\Tc\times\Hc_\Tc)$ as the space of coarse
variables, and we can take face values, denoted $(\Hbc_\Fc\times
\Hc_\Fc)=(\Hbc_\Dc\times\nobreak\Hc_\Dc)\setminus(\Hbc_\Tc\times\Hc_\Tc)$,
as the space of fine variables, we thus consider the problem: Find
$((\ub_\Tc,p_\Tc),(\ub_\Fc,p_\Fc),(\yb_\Cc,y_\Cc),\yb_\Dc)\in(\Hbc_\Tc\times
\Hc_\Tc)\times (\Hbc_\Fc\times \Hc_\Fc)\times (\Hbc_\Cc\times \Hc_\Cc)\times
(\Hbc_\Dc\times \Hc_\Dc)$ such that:

\begin{alignat}{2}
\intertext{\em Coarse problem:}
&a_\Dc (\{\ub_\Tc,\ub_\Fc\},\vb)=-\int_\Omega \fb_{\ub}\cdot\vb \,d\xb
&\quad&\text{for all } \vb\in\Hbc_\Tc\label{eq4.1}\\
&b_{\Dc,1} (\{\ub_\Tc,\ub_\Fc\},r)
-[\rho_\Tc p_\Tc,r]_{\Tc,0}\label{eq4.2}\\
&\quad -\tau c_\Dc (\{p_\Tc,p_\Fc\},r)=-\int_\Omega f_p\cdot r \,d\xb
&\quad&\text{for all } r\in\Hc_\Tc\notag
\displaybreak[1]
\intertext{\em Fine problems:}
&a_\Dc (\{\Ob,\ub_\Fc\},\vb)\label{eq4.3}\\
&\qquad=-a_\Dc (\{\ub_\Tc,\Ob\},\vb)-b_{\Dc,2}^T(\{p_\Tc,0\},\vb)
&\quad&\text{for all } \vb\in\Hbc_\Fc\cap\Hbc_\Cc\notag\\
&c_\Dc (\{0,p_\Fc\},r)=-c_\Dc (\{p_\Tc,0\},r)
&\quad&\text{for all } r\in\Hc_\Fc\cap\Hc_\Cc\label{eq4.4}\\
&d_\Dc \bigl((\{\Ob,\ub_\Fc\},\{0,p_\Fc\}),\wb\bigr)\label{eq4.5}\\
&\qquad
=-d_\Dc\bigl((\{\ub_\Tc,\Ob\},\{p_\Tc,0\}),\wb\bigr)
&\quad&\text{for all } \wb\in(\Hbc_\Dc\times \Hc_\Dc)\notag\\
&e_\Dc \bigl((\{\Ob,\ub_\Fc\},\{0,p_\Fc\}),(\wb,w)\bigr)
+ \hat{a}_\Dc (\wb,\yb_\Cc)\label{eq4.6}\\ 
&\quad+ b_\Dc (\wb,y_\Cc) + b_\Dc^T (w,\yb_\Cc)\notag\\
&\quad-c(w,y_\Cc) + d_\Dc \bigl((\wb,w),\yb_\Dc\bigr)\notag\\
&\qquad 
=-e_\Dc\bigl((\{\ub_\Tc,\Ob\},\{p_\Tc,0\}),(\wb,w)\bigr)
&\quad&\text{for all } (\wb,w)\in(\Hbc_\Dc\times \Hc_\Dc)\notag
\end{alignat}
In the above equations, we have abused notation slightly be letting
``$0$'' represent elements in either spaces $\Hc_\Tc$ or $\Hc_\Fc$,
however the meaning should be clear from the context. Furthermore,
we note the following important details which arise as a consequence
of the previous definitions: There is no coupling term in the
coarse equation~\eqref{eq4.1}, since by Gauss' theorem $b_{\Dc,2}^T
(p_\Dc,\vb)=0$ for all $\vb\in\Hbc_\Tc$. Similarly, by the definition
in equation~\eqref{eq3.15}, there are no coupling terms $b_{\Dc,1}$
in equation~\eqref{eq4.4}, and only the coarse component appears in
equation~\eqref{eq4.3}. Finally, there is no source term on the right
hand side of equations \eqref{eq4.3} and \eqref{eq4.4} since $\Pi_\Tc
(\Hc_\Fc\cap\Hc_\Cc)=0$.

The discrete variational multiscale formulation given above allows us
to define a fine-scale interpolation operator for the finite volume
method: $\Pi_{\FV}\colon\Hbc_\Tc\times \Hc_\Tc\to \Hbc_\Dc\times
\Hc_\Dc$. Furthermore, due to the lack of coupling terms in
Equation~\eqref{eq4.4}, and the linearity of the system, we note that
this interpolation operator can be decomposed as
\begin{equation}
\Pi_{\FV} \{\ub_\Tc,p_\Tc\}=\{\Pi_{\FV}^{\ub,\ub} \ub_\Tc+\Pi_{\FV}^{\ub,p}
p_\Tc,\Pi_{\FV}^p p_\Tc\}
\label{eq4.7}
\end{equation}
Here, the fluid interpolation operator $\Pi_{\FV}^p$ is (due to the lack of
coupling in equation~\eqref{eq4.4}) identical to the standard operator
for the uncoupled system. In the MPFA terminology this is the solution
operator for the interaction region problem \cite{aav02}. Similarly,
and again due to linearity, it is easily verified that the solid operator
$\Pi_{\FV}^{\ub,\ub}$ is also identical to the finite volume interpolation
operator for the uncoupled problem \cite{nor15}. The critical novel
component is the influence of the cell-center pressures on the sub-scale
structure of the displacement, given by $\Pi_{\FV}^{\ub,p}$. We will
find below that \emph{this operator is essential for the consistency and
stability of the method}, and would be neglected if simply introducing
the Biot coupling at the coarse scale \cite{nor14b}.

The following lemma is establishes the computational applicability of
the method.

\begin{lemma}\label{lem4.1}
The finite volume interpolation $\Pi_{\FV}$ is local, in the sense that
it can be decomposed as
\begin{equation}
\Pi_{\FV}=\sum_{s\in\Vc}\Pi_{\FV,s}
\label{eq4.8}
\end{equation}
Where furthermore for each $s\in\Vc$, the operator $\Pi_{\FV,s}$ depends
only on elements in $\Hbc_{\Tc_s}\times \Hc_{\Tc_s}$.
\end{lemma}

\begin{proof}
Let $\chi_\Dc$ be any of the bilinear forms defined in
\eqref{eq3.12}--\eqref{eq3.17}. We note that in all definitions, the sums
can be rearranged such that
\begin{equation}
\chi_\Dc=\sum_{s\in\Vc}\chi_{\Dc,s}
\label{eq4.9}
\end{equation}

By inspection, we see that all bilinear forms $\chi_{\Dc,s}$ have local
support, involving only elements in $\Hbc_{\Tc_s}\times \Hc_{\Tc_s}$
and variables $(\ub,p)_{K,s'}^{\sigma,\beta}\in\Hbc_\Fc\times \Hc_\Fc$
with $s'=s$. Hence, the local systems \eqref{eq4.3}--\eqref{eq4.6} form a
block-diagonal system with respect to $\Hbc_\Fc\times \Hc_\Fc$, and can be
solved locally for each vertex.
\end{proof}

\begin{lemma}\label{lem4.2}
The local systems given by equations \eqref{eq4.3}--\eqref{eq4.6} have
a unique solution $(\ub_\Fc,p_\Fc)\in(\Hbc_\Fc\times \Hc_\Fc)$ for each
$(\ub_\Tc,p_\Tc)\in\Hbc_\Tc\times \Hc_\Tc$.
\end{lemma}

\begin{proof}
The unique solvability for both the fluid and solid problem follows from
coercivity arguments as given in \cite{nor15}.
\end{proof}

The existence and uniqueness of the solution defines the locally
computable finite volume interpolations $\Pi_\FV$, and thus we define:

\begin{definition}\label{def4.3}
For every $s\in\Vc$, the local mixed problem \eqref{eq4.3}--\eqref{eq4.6}
has a unique solution by Lemma~\ref{lem4.2}, and we define the norm of
the of the solution operator $\Pi_{\FV,s}$ as $\theta_{1,s}$ such that
for all $(\ub,p)\in(\Hbc_\Tc\times \Hc_\Tc)$:
\begin{equation}
\norm{\Pi_{\FV,s}^{\ub,\ub}\ub}_{\Dc,s} + h^{-1} \norm{\Pi_{\FV,s}^{\ub,p}
p}_{\Dc,s}
+\norm{\Pi_{\FV,s}^p p}_{\Dc,s}\le \theta_{1,s}
(\norm{\ub}_{\Tc,s}+\norm{p}_{\Tc,s})
\label{eq4.10}
\end{equation}
\end{definition}

It follows from the definition of norms and the scaling invariance of
equations \eqref{eq4.3}--\eqref{eq4.6} that the constants $\theta$ may depend
on heterogeneity, grid geometry, but not on domain size or mesh size. In
particular, the appearance of $h^{-1}$ is due to the lack of a derivative
on $p_\Dc$ in the bilinear form $b_{\Dc,2}^T$ in equation~\eqref{eq4.3}.

\subsection{Cell-centered discretization for Biot's
equations}\label{sec4.2}

We are now prepared to state the cell-centered discretization
of Biot's equations. Indeed, by replacing the fine-scale
terms in equation~\eqref{eq4.1}--\eqref{eq4.2} by the finite volume
interpolation $\Pi_\FV$, we obtain a finite cell-centered finite volume
discretization. From section~\ref{sec4.1} we thus obtain the following
coarse problem: Find $(\ub_\Tc,p_\Tc)\in\Hbc_\Tc\times \Hc_\Tc$
\begin{alignat}{2}
&a_\Dc (\Pi_{\FV}^{\ub,\ub} \ub_\Tc,\vb)+a_\Dc (\Pi_{\FV}^{\ub,p}
p_\Tc,\vb)=-\int_\Omega\fb_{\ub}\cdot\vb\,d\xb
&\quad&\text{for all } \vb\in\Hbc_\Tc
\label{eq4.11}\\
&b_{\Dc,1} (\Pi_{\FV}^{\ub,\ub} \ub_\Tc,r)-[\rho_\Tc p_\Tc,r]_{\Tc,0}\label{eq4.12}\\
&\quad-\tau c_\Dc (\Pi_{\FV}^p p_\Tc,r)
+b_{\Dc,1} (\Pi_{\FV}^{\ub,p} p_\Tc,r)
=-\int_\Omega f_p\cdot r\,d\xb
&\quad&\text{for all }r\in\Hc_\Tc
\notag
\end{alignat}
Importantly, in equations \eqref{eq4.11}--\eqref{eq4.12}, two
non-standard terms have appeared due to the coupling between fluid
pressure and deformation captured by $\Pi_{\FV}^{\ub,p}$. The term $a_\Dc
(\Pi_{\FV}^{\ub,p} p_\Tc,\vb)$ in equation~\eqref{eq4.11} will be seen
to be an approximation of the continuous operator $(p,\nabla\cdot\vb)$,
and thus provides the correct impact of pressure in the mechanical
equation. In contrast, the term $b_{\Dc,1} (\Pi_{\FV}^{\ub,p} p_\Tc,r)$
in equation~\eqref{eq4.12} is a local consistency operator which
approximates of the sub-scale impact of pressure gradients on local
volume changes.

For convenience of notation, we will by define by capital letters indicate
the bilinear forms with the finite volume interpolations suppressed,
such that e.g.\ $A_\Dc (\ub_\Tc,\vb)=a_\Dc (\Pi_{\FV}^{\ub,\ub}
\ub_\Tc,\vb)$. The exception to this rule is the coupling term,
which for notational consistency is denoted and $B_{\Dc,2}^T=a_\Dc
(\Pi_{\FV}^{\ub,p} p_\Tc,\vb)$, the local consistency operator which is
denoted by $\Delta_\Dc (p_\Tc,r)=b_\Dc (\Pi_{\FV}^{\ub,p} p_\Tc,r)$. This
allows us to define our cell-centered discretizations compactly.

\begin{definition}\label{def4.4}
We then refer to the following system as the \emph{MPSA/MPFA
finite volume discretization} for Biot's equations: Find
$(\ub_\Tc,p_\Tc)\in\Hbc_\Tc\times \Hc_\Tc$
\begin{alignat}{2}
&A_\Dc (\ub_\Tc,\vb)+B_{\Dc,2}^T (p_\Tc,\vb)=-\int_\Omega \fb_{\ub}\cdot
\vb \,d\xb
&\quad&\text{for all } \vb\in\Hbc_\Tc
\label{eq4.13}\\
&B_{\Dc,1} (\ub_\Tc,r)-[\rho_\Tc p_\Tc,r]_{\Tc,0}\label{eq4.14}\\
&\quad-\tau C_\Dc (p_\Tc,r)+\Delta_\Dc (p_\Tc,r)=-\int_\Omega f_p\cdot r \,d\xb
&\quad&\text{for all } r\in\Hc_\Tc
\notag
\end{alignat}
We will refer to this system in shorthand as
\begin{equation}
\Afr_\Dc (\ub_\Tc,p_\Tc,\vb,r)=\Bfr(\vb,r)
\quad\text{for all } (\vb,r)\in\Hbc_\Tc\times \Hc_\Tc
\label{eq4.15}
\end{equation}
Where to be precise, we may sometimes explicitly include the
parameter dependencies in the bilinear form, $\Afr_\Dc=\Afr_\Dc
(\ub_\Tc,p_\Tc,\vb,r;\rho,\tau)$.
\end{definition}

{\makeatletter
\def\@begintheorem#1#2{\par\bgroup{{\itshape #1}\ #2. }\ignorespaces}
\makeatother
\begin{remark}\label{rem4.5}
Due to the static condensation, the \emph{local consistency operator}
$\Delta_\Dc (p_\Tc,r)$ has appeared in the fluid pressure system. Since
$\Pi_{\FV}^{\ub,p} p_\Tc$ expresses a displacement response to pressure,
and $b_{\Dc,2}$ evaluates the discrete divergence, and by scaling
arguments, qualitatively identify the new term as proportional to a weak
discretization of $h^2 \frac{\partial}{\partial t}\nabla\cdot (\alpha
\kappa^{-1} \nabla(\alpha p))$, where $\kappa$ is the bulk modulous of
the material. Terms with similar scaling have previously been introduced
artificially in order to obtain a stable discretization of equations
\eqref{eq1.1} (see e.g.\ \cite{bre91,gas08}). We note that this term
is essential for the stability method of the current scheme, although
numerical experiments indicate that it is not sufficient to guarantee
monotonicity of the resulting pressure solution.{\emergencystretch5pt\par}
\end{remark}

\begin{remark}\label{rem4.6}
For analysis, it will be desirable to control the asymmetry of the
coupling terms. However, we note that this is not possible on the
current form, since $B_{\Dc,1} (\ub_\Tc,r)\sim\norm{\ub_\Tc}_\Tc
\norm{r}_{\Tc,0}$ while $B_{\Dc,2}^T (p_\Tc,\vb)\sim\norm{p_\Tc}_\Tc
\norm{\vb}_{\Tc,0}$. However, we exploit equation~\eqref{eq4.3} to
obtain the relationship for all $(\vb,p_\Tc)\in\Hbc_\Tc\times \Hc_\Tc$
\begin{align}
a_\Dc (\Pi_{\FV}^{\ub,p} p_\Tc,\vb)
&=a_\Dc (\Pi_{\FV}^{\ub,p} p_\Tc,\Pi_\Tc \Pi_{\FV}^{\ub,\ub}\vb)\label{eq4.16}\\
&=a_\Dc (\Pi_{\FV}^{\ub,p} p_\Tc,\Pi_\Cc \Pi_{\FV}^{\ub,\ub}\vb)
+b_{\Dc,2}^T (p_\Tc,\Pi_\Cc \Pi_{\FV}^{\ub,\ub}\vb)
\notag
\end{align}
Then by the definition of the bilinear forms, we have that
\begin{align}
B_{\Dc,2}^T (p_\Tc,\vb)=B_{\Dc,1}^T (p_\Tc,\vb)
&+a_\Dc (\Pi_{\FV}^{\ub,p} p_\Tc,\Pi_\Cc \Pi_{\FV}^{\ub,\ub}\vb)\label{eq4.17}\\
&+[b_{\Dc,2}^T (p_\Tc,\Pi_\Cc \Pi_{\FV}^{\ub,\ub}\vb)-b_{\Dc,1}
(\Pi_{\FV}^{\ub,\ub}\vb,p_\Tc)]
\notag
\end{align}

In the continuation, we will denote the asymmetric terms as
\begin{equation}
\Lambda_\Dc (p_\Tc,\vb)=a_\Dc (\Pi_{\FV}^{\ub,p} p_\Tc,\Pi_\Cc
\Pi_{\FV}^{\ub,\ub}\vb)
+[b_{\Dc,2}^T (p_\Tc,\Pi_\Cc \Pi_{\FV}^{\ub,\ub}\vb)-b_{\Dc,1}
(\Pi_{\FV}^{\ub,\ub}\vb,p_\Tc)]
\label{eq4.18}
\end{equation}
Such that equation~\eqref{eq4.13} can be equivalently written as
\begin{equation}
A_\Dc (\ub_\Tc,\vb)+B_{\Dc,1}^T (p_\Tc,\vb)+\Lambda_\Dc
(p_\Tc,\vb)=-\int_\Omega \fb_{\ub}\cdot \vb \,d\xb
\quad\text{for all } \vb\in\Hbc_\Tc
\label{eq4.19}
\end{equation}

While equation~\eqref{eq4.13} and \eqref{eq4.19} are formally
identical, we note that due to the conservation structure, equations
\eqref{eq4.13}--\eqref{eq4.14} are natural from the perspective of
method formulation and implementation. Conversely, as $\Lambda_\Dc
(p_\Tc,\vb)$ can be bounded by scaling arguments as seen below, equations
\eqref{eq4.14} and \eqref{eq4.19} are more convenient for analysis.
\end{remark}
}

\section{Convergence}\label{sec5}

In the introduction, we noted that the three main concerns for
discretizing the Biot equations are stability with respect to
incompressible materials, as well as small time-steps, which is
represented Definition~\ref{def1.1}. In this section, we will show
stability and convergence of the discretization, independent of small
parameters, dependent only on locally computable conditions on the grid.

\subsection{Stability of uncoupled problems}\label{sec5.1}

It is well known that finite volume methods of the type considered
herein are stable and convergent for a wide range of grids
\cite{kla06,age10,eig05}. In particular, we will make use of the following
mild local condition on the mesh and parameters:

{\makeatletter
\def\@begintheorem#1#2{\par\bgroup{{\itshape #1}\ #2. }\ignorespaces}
\makeatother
\begin{condition}\label{conA}
For every vertex $s\in\Vc$, there exists a constant
$\theta_s^c\ge\theta^c>0$ such that the bilinear form $c_{\Dc,s}$ and the
interpolation $\Pi_{\FV,s}^p$ satisfy for all $p\in\Pi_{\FV,s}^p \Hc_\Tc$
\begin{equation}
c_{\Dc,s} (p,\Pi_\Cc p)\ge\theta_s^c \biggl(\abs{p}_{c_{\Dc,s}}^2
+\sum_{K\in\Tc_s}\sum_{\sigma
\in\Fc_s\cap\Fc_K}\frac{m_K^s}{d_{K,\sigma}^2}
\frac1{m_s^\sigma} \sum_{\beta\in\Gc_s^\sigma}\omega_\beta
\bigl(\dblbr{p}_s^{\sigma,\beta}\bigr)^2\biggr)
\label{eq5.1}
\end{equation}
Where the local energy semi-norm is associated with the symmetrized
bilinear form
\[
\abs{p}_{c_{\Dc,s}}^2=\sum_{K\in\Tc_s}m_K^s k_K ((\bnabla p)_K^s)^2
\]
\end{condition}

For the solid deformation problem, a similar local coercivity condition
is needed.

\begin{condition}\label{conB}
For every vertex $s\in\Vc$, there exists a constant
$\theta_s^a\ge\theta^a>0$ such that the bilinear form $a_{\Dc,s}$
and the interpolation $\Pi_{\FV,s}^{\ub,\ub}$ satisfy for all
$\ub\in\Pi_{\FV,s}^{\ub,\ub} \Hbc_\Tc$
\begin{equation}
a_{\Dc,s} (\ub,\Pi_\Cc \ub)\ge\theta_s^a \biggl(\abs{\ub}_{a_{\Dc,s}}^2
+\sum_{K\in\Tc_s}\sum_{\sigma
\in\Fc_s\cap\Fc_K}\frac{m_K^s}{d_{K,\sigma}^2}\frac1{m_s^\sigma}
\sum_{\beta\in\Gc_s^\sigma}\omega_\beta \bigl(\dblbr{\ub}_s^{\sigma,\beta}\bigr)^2\biggr)
\label{eq5.2}
\end{equation}
Where the local energy semi-norm is associated with the symmetrized
bilinear form
\[
\abs{\ub}_{a_{\Dc,s}}^2=\sum_{K\in\Tc_s}m_K^s
(\bnabla\ub)_K^s:\C:(\bnabla\ub)_K^s
\]
\end{condition}
}

Conditions~\ref{conA} and \ref{conB} are in practice a condition on the
grid regularity and the structure (but not magnitude) of the material
properties. These conditions can be verified locally while assembling
the discretization, and moreover they can be verified \emph{a priori}
for certain classes of meshes \cite{nor15}.

The known stability results are recalled in the following lemma (see
\cite{age10,nor15}):

\begin{lemma}\label{lem5.3}
For given parameter fields $\C$, and mesh $\Dc$, let condition~\ref{conA}
and \ref{conB} hold. Then the bilinear forms $A_\Dc$ and $C_\Dc$ are
coercive and for all $\ub\in\Hbc_\Tc$ and $p\in\Hc_\Tc$ there exists
positive constants such that $\Theta^A$ and $\Theta^C$ such that
\begin{equation}
A_\Dc (\ub,\ub)\ge\Theta^A \norm{\ub}_\Tc^2
\quad\text{and}\quad
(\rho_\Tc p,p)+\tau C_\Dc (p,p)\ge\rho \norm{p}_{\Tc,0}^2+\tau \Theta^C
\norm{p}_\Tc^2
\label{eq5.3}
\end{equation}
\end{lemma}

The constants $\Theta^A$ and $\Theta^C$ are dependent on the parameters
of the problem and the mesh triplet $\Dc$ through the constants $\theta^a$
and $\theta^c$, but do not scale with $h$. Here, and in the continuation,
we will use the convention on material parameters from section~\ref{sec1}
such that $\rho =\sup_{K\in\Tc}\rho_K$.

\subsection{Properties of the non-standard terms}\label{sec5.2}

To proceed, we will need to verify that the local consistency operator
$\Delta_\Dc$ is stable, and identify a bound on the properties of
$\Lambda_\Dc$. We first identify a local coercivity condition on the
mesh similar to those stated above.

{\makeatletter
\def\@begintheorem#1#2{\par\bgroup{{\itshape #1}\ #2. }\ignorespaces}
\makeatother
\begin{condition}\label{conC}
For every internal vertex $s\in\Vc$ with an associated local mesh
diameter $h_s=\max_{K\in\Tc_s}d_K$, there exists a constant
$\theta_s^\Delta\ge\theta^\Delta>0$, such that the bilinear form
$b_{\Dc,s}$ and the interpolation $\Pi_{\FV,s}^{\ub,p}$ satisfy for all
$p\in\Hc_\Tc$
\begin{equation}
b_{\Dc,s} (\Pi_{\FV,s}^{\ub,p} p,p)\ge h_s^2 \theta_s^\Delta
\biggl(\abs{p}_{\Delta_{\Dc,s}}^2+\sum_{K\in\Tc_s}\sum_{\sigma \in\Fc_s\cap\Fc_K}
\frac{m_K^s}{d_{K,\sigma}^2}\frac1{m_s^\sigma}
\sum_{\beta\in\Gc_s^\sigma}\omega_\beta \bigl(\dblbr{p}_s^{\sigma,\beta}\bigr)^2\biggr)
\label{eq5.4}
\end{equation}
Where the local energy semi-norm is associated with the symmetrized
bilinear form
\[
\abs{p}_{\Delta_{\Dc,s}}^2=\sum_{K\in\Tc_s}m_K^s ((\bnabla p)_K^s)^2
\]
\end{condition}
}

As with condition~\ref{conA} and \ref{conB}, the local structure of
Condition~\ref{conC} allows it to be easily verified \emph{a priori}
for grids with local symmetry using similar arguments as found in
\cite{nor15}, and can be verified numerically at the discretization
stage for arbitrary grids. The structure of this condition, and in
particular the scaling $h_s^2$, follows from equation~\eqref{eq4.10}
and relationship between the norms.

\begin{lemma}\label{lem5.4}
For given parameter fields $\C$, and mesh $\Dc$, and let
Condition~\ref{conC} hold, then the bilinear form $\Delta_\Dc$ is coercive
and for all $p\in\Hc_\Tc$ there exists a positive constant $\Theta^\Delta$
such that
\begin{equation}
\Delta_\Dc (p,p)\le -h^2 \Theta^\Delta \abs{p}_\Tc^2
\label{eq5.5}
\end{equation}
The constant $\Theta^\Delta$ is dependent on the mesh triplet $\Dc$
through the constant $\theta^\Delta$, but does not scale with $h$.
\end{lemma}

\begin{proof}
The lemma follows by summation of equation~\eqref{eq5.4}, the
well-posedness of the local problems, and scaling arguments. Note that
we do not get a full norm, since $\Pi_{\FV}^{\ub,p} p_\Tc$ does not depend
on the Dirichlet boundary data on $\Gamma_{p,D}$.
\end{proof}

Finally, we bound the non-symmetric part of the coupling.

\begin{definition}\label{def5.5}
For every vertex $s\in\Vc$ with an associated local mesh
diameter $h_s=\max_{K\in\Tc_s}d_K$, there exists two
constants $\theta_1^\Lambda\ge\theta_{1,s}^\Lambda>0$ and
$\theta_2^\Lambda\ge\theta_{2,s}^\Lambda>0$, such that the bilinear
form $b_{\Dc,s}$ and the interpolation $\Pi_{\FV,s}^{\ub,p}$ satisfy for
all $p\in\Hc_\Tc$
\begin{equation}
a_{\Dc,s} (\Pi_{\FV}^{\ub,p} p_\Tc,\Pi_\Cc \Pi_{\FV}^{\ub,\ub}\vb)
\ge-h_s \theta_{1,s}^\Lambda \abs{p}_{\Tc,s} \norm{\vb}_{\Tc,s}
\label{eq5.6}
\end{equation}
And
\begin{equation}
\bigl[b_{\Dc,2,s}^T (p_\Tc,\Pi_\Cc \Pi_{\FV}^{\ub,\ub}\vb)
-b_{\Dc,1,s}(\Pi_{\FV}^{\ub,\ub}\vb,p_\Tc)\bigr]
\ge-\theta_{2,s}^\Lambda \norm{p}_{\Tc,0,s} \norm{\vb}_{\Tc,s}
\label{eq5.7}
\end{equation}
\end{definition}

We note that the existence of these constants follows from scaling
arguments and the linearity of the introduced operators.

\begin{lemma}\label{lem5.6}
There exists a lower bound on the bilinear forms $\Lambda_\Dc$ is denoted
$\Theta^\Lambda$, such that the following inequality holds:
\begin{equation}
\Lambda_\Dc (\vb,p)\ge-\Theta^\Lambda \norm{p}_{\Tc,0} \norm{\vb}_\Tc
\quad\text{for all }p\in\Hc_\Tc\text{ and }\vb\in\Hbc_\Tc
\label{eq5.8}
\end{equation}
Furthermore, the constant $\Theta^\Lambda\le\max(d^{-1/2}
\theta_1^\Lambda,\theta_2^\Lambda)$.
\end{lemma}

\begin{proof}
To show equation~\eqref{eq5.8}, we start from the definition of the
bilinear form, and by the triangle inequality we have that
\begin{align*}
\Lambda_\Dc (\vb,p)&=\sum_{s\in\Vc}a_{\Dc,s} (\Pi_{\FV}^{\ub,p} p_\Tc,\Pi_\Cc
\Pi_{\FV}^{\ub,\ub}\vb)
+\bigl[b_{\Dc,2,s}^T (p_\Tc,\Pi_\Cc \Pi_{\FV}^{\ub,\ub}\vb)-b_{\Dc,1,s}
(\Pi_{\FV}^{\ub,\ub}\vb,p_\Tc)\bigr]\\
&\ge -\sum_{s\in\Vc}\bigl[h_s \theta_{1,s}^\Lambda \abs{p}_{\Tc,s}
\norm{\vb}_{\Tc,s}
+\theta_{2,s}^\Lambda \norm{p}_{\Tc,s,0} \norm{\vb}_{\Tc,s}\bigr]\\
&\ge -\Theta^\Lambda \norm{p}_{\Tc,0} \norm{\vb}_\Tc
\end{align*}
Here the last inequality uses the inverse inequality stated in
section~\ref{sec2}.
\end{proof}

\subsection{Stability of coupled system}\label{sec5.3}

Let us first note that the simplest approaches to showing the stability
of the coupled system is not adequate. Indeed, from \eqref{eq4.14}
and \eqref{eq4.19}, we immediately obtain a stability estimate for
the coupled problem based on Lemma~\ref{lem5.3}, \ref{lem5.4} and
\ref{lem5.6}, e.g.
\begin{align}
\Afr_\Dc (\ub,p,\ub,-p)\ge\Theta^A \norm{\ub}_\Tc^2
&+\tau \Theta^C \norm{p}_\Tc^2\label{eq5.9}\\ 
&+ h^2 \Theta^\Delta \abs{p}_\Tc^2
+\rho \norm{p}_{\Tc,0}^2-\Theta^\Lambda \norm{p}_{\Tc,0}
\norm{\ub}_{\Tc,0}
\notag
\end{align}
This estimate is unsatisfactory, as it both contains a negative term,
and furthermore there is no bound on pressure as $\rho,\tau,h\to 0$. To
remedy this situation, we must explicitly utilize the properties of the
coupling term $B_{\Dc,1}$, which is usually achieved through showing that
an inf-sup condition holds \cite{bre91}. In our case this is however
not possible the kernel of $B_{\Dc,1}$ admits non-trivial oscillatory
solutions. The most notable example arises on square grids, where if
the pressure $p$ is a so-called checker-board pattern, when $B_{\Dc,1}
(p,\vb)=0$ for all $\vb\in\Hbc_\Dc$. This is a fundamental consequence
of Galilean invariance of the discretization, and is common for all
low-order discretizations with co-located variables \cite{bre91,lew98}.

To prove stability of the system, we must therefore exploit the
structure of the full coupled system. Intuitively, the observation
that non-oscillatory solutions are well captured by $B_{\Dc,1}
(p,\vb)$, while oscillatory solutions have a more favorable bound
than that given in Lemma~\ref{lem5.3}, allows us to expect that the
presence of the local consistency operator, which we have qualitatively
identified as a Laplacian-like term, proportional to $h^2$, may be
sufficient to stabilize the system. Indeed, this is the idea behind
several stabilization techniques, and as applied to finite element
discretizations for Stokes \cite{bre84}, and also to finite element
\cite{bad09} and finite difference \cite{gas06} discretizations of Biot's
equation. The main difference between the cited works and the present,
is that in previous work a stabilization was introduced explicitly or
through augmented formulations, herein the term arises as a natural part
of the discretization.

To prove that the system $\Afr_\Dc$ is indeed stable, our analysis will
follow a similar path do that of Franca and Stenberg, who analyzed
finite element discretizations of augmented formulations for mixed
form of the elasticity equations \cite{fra91}. The main idea is to
use the properties of $B_{\Dc,1}$ to trade the $\abs{p}_\Tc^2$ term in
equation~\eqref{eq5.9} for a deficiency in the inf-sup condition. Thus
we do not have an inf-sup condition in the sense of Brezzi \cite{bre74},
but rather derive inf-sup for the global system in the sense of Babuska
\cite{bab71}. We will need some technical results. In the following,
$C$ and $c$ represent generic positive constants, independent of $h$.

Let us first introduce an interpolation operator from the discrete
to continuous setting, in the spirit of a discontinuous Galerkin
interpretation.

\begin{definition}\label{def5.8}
Let the interpolation operator $\Pi_{L^2}\colon\Hbc_\Dc\to (L^2
(\Omega))^d$ be defined such that for all $(K,s)\in\Tc\times\Vc$ and $\xb\in
K$, we have $\Pi_{L^2} \ub(\xb)=\ub_K+(\bnabla\ub)_K^s\cdot (\xb-\nobreak\xb_K)$.
\end{definition}

\begin{lemma}\label{lem5.9}
For all $p\in\Hc_\Tc$ it holds that
\[
\sup_{\wb\in \Hb^1}\frac{(p,\nabla\cdot\wb)}{\norm{\wb}_1} \ge c_{\LBB}
\norm{p}_{\Tc,0}
\]
\end{lemma}

\begin{proof}
Since $\Hc_\Tc\in L^2$, the result follows from the well-posedness of
the continuous problem, and the constant $c_{\LBB}$ is no worse than the
inf-sup constant for the continuous problem.
\end{proof}

\begin{lemma}\label{lem5.10}
For all $\wb\in\Hb^1$ there exists $\widetilde{\wb}\in\Hbc_\Tc$ such that
\[
\norm{\widetilde{\wb}}_\Tc\le \norm{\wb}_1
\]
And
\[
\norm{\Pi_\Dc\wb-\Pi_{\FV}^{\ub,\ub} \widetilde{\wb}}_{\Dc,0}\le
\ch\norm{\wb}_1
\]
\end{lemma}

\begin{proof}
The result follows from the stability of the local calculations
(Lemma~\ref{lem4.2} and Definition~\ref{def4.3}) and scaling arguments
\cite{kla06,eym06}.
\end{proof}

We now consider the following weakened LBB-type condition, the proof
of which follows closely the approach used in reference \cite{fra91},
but is adapted to the discrete norms used herein, and the finite spaces
$\Pi_{L^2} \Pi_{\FV}^{\ub,\ub} \Hbc_\Tc$.

\begin{lemma}\label{lem5.11}
For sufficiently fine grids, the bilinear form $B_{\Dc,1}$ satisfies
\begin{equation}
\sup_{\substack{\vb\in\Hbc_\Tc\\\norm{\vb}_\Tc=1}} B_{\Dc,1} (\vb,p)
\ge\Theta^B \abs{p}_{\Tc,0}-\theta^B h\abs{p}_\Tc
\quad\text{for all } p\in\Hc_\Tc
\label{eq5.10}
\end{equation}
\end{lemma}

\begin{proof}
Now for any $p\in\Hc_\Tc$, consider the splitting such that $p=\bar{p}
+ \widetilde{p}$ where $\bar{p}=\abs{\Omega}^{-1} \sum_{K\in\Tc}m_K
p_K$. Then by Lemma~\ref{lem5.9} there exists $\wb\in\Hbc_0^1$ such that
\begin{equation}
(\nabla\cdot\wb,\widetilde{p})
\ge c_{\LBB} \norm{\widetilde{p}}_{\Tc,0}\norm{\wb}_1
=c_{\LBB}\abs{p}_{\Tc,0}\norm{\wb}_1
\ge c_{\LBB} \abs{p}_{\Tc,0} \norm{\widetilde{\wb}}_\Tc
\label{eq5.11}
\end{equation}
Where the function $\widetilde{\wb}\in\Hbc_\Tc$ is as defined in
Lemma~\ref{lem5.10}. Then we calculate
\begin{align}
B_{\Dc,1} (-\widetilde{\wb},\widetilde{p})
&=\sum_{K\in\Tc}\widetilde{p}_K \sum_{s\in\Vc_K}m_K^s
(\widetilde{\nabla}\cdot \Pi_{\FV}^{\ub,\ub} \widetilde{\wb})_K^s\label{eq5.12}\\
&=\sum_{K\in\Tc}\sum_{s\in\Vc_K}\sum_{\sigma
\in\Fc_K\cap\Fc_s}\widetilde{p}_K
\aver{\Pi_{\FV}^{\ub,\ub} \widetilde{\wb}-\Pi_\Dc \wb}_{K,s}^\sigma
\cdot \gb_{K,\sigma}^s\notag\\ 
&\qquad +\sum_{K\in\Tc}\bigl(\nabla\cdot (\Pi_{L^2}
\Pi_\Dc\wb-\wb),p\bigr)_K +(\nabla\cdot\wb,\widetilde{p})\notag\\
&\ge\sum_{K\in\Tc}\sum_{s\in\Vc_K}\sum_{\sigma
\in\Fc_K\cap\Fc_s}\widetilde{p}_K
\aver{\Pi_{\FV}^{\ub,\ub} \widetilde{\wb}-\Pi_\Dc \wb}_{K,s}^\sigma
\cdot \gb_{K,\sigma}^s\notag\\ 
&\qquad-\ch\abs{p}_{\Tc,0} \norm{\wb}_1+c_{\LBB}
\abs{p}_{\Tc,0} \norm{\wb}_1
\notag
\end{align}
To estimate the remaining summation, we consider jumps across internal
edges, exploiting that $\gb_{K,\sigma}^s=-\gb_{L,\sigma}^s$ for
$\{K,L\}=\Tc_\sigma$. For simplicity, we introduce $\xib_{K,s}^\sigma
=\aver{\Pi_{\FV}^{\ub,\ub} \widetilde{\wb}-\Pi_\Fc \wb}_{K,s}^\sigma$,
after which
\begin{multline}
\sum_{s\in\Vc}\sum_{\sigma \in\Fc_s}\sum_{K\in\Tc_\sigma}\widetilde{p}_K
\aver{\Pi_{\FV}^{\ub,\ub} \widetilde{\wb}-\Pi_\Fc \wb}_{K,s}^\sigma \cdot
\gb_{K,\sigma}^s\\
=\sum_{s\in\Vc}\sum_{\sigma \in\Fc_s}\biggl(\aver{\xib}_s^\sigma +
\frac{\dblbr{\xib}_s^\sigma}2\biggr)
\biggl(\aver{\widetilde{p}\gb_{K,\sigma}^s}_s^\sigma
+\frac{\dblbr{\widetilde{p}\gb_{K,\sigma}^s }_s^\sigma}2\biggr)
\ge-\ch\norm{\wb}_1 \abs{p}_\Tc
\label{eq5.13}
\end{multline}
Combining equations \eqref{eq5.12}--\eqref{eq5.13} provides
\begin{equation}
\frac{B_{\Dc,1}
(-\widetilde{\wb},\widetilde{p})}{\norm{\widetilde{\wb}}_\Tc}
\ge(c_{\LBB}-\ch)\abs{p}_{\Tc,0}-\ch\abs{p}_\Tc
\label{eq5.14}
\end{equation}
Finally, to treat the constant $\bar{p}$ it is sufficient that there
exists a function $\zb\in\Hbc_\Tc$ such that
\begin{equation}
B_{\Dc,1} (\zb,1)\ge c\norm{\widetilde{\wb}}_\Tc
\label{eq5.15}
\end{equation}
This such a function trivially exists, thus combining equations
\eqref{eq5.14} and \eqref{eq5.15}, proves the lemma.
\end{proof}

Using Lemma~\ref{lem5.11}, we can now show the stability of the
MPSA/MPFA-FV discretization for Biot's equations.

\begin{lemma}\label{lem5.12}
Let Conditions~\ref{conA}, \ref{conB} and \ref{conC} hold, and
let the grid be sufficiently fine such that Lemma~\ref{lem5.11}
applies. Furthermore, let the asymmetry of the method be moderate in
the sense that $8\Theta^\Lambda<\Theta^B$. Then the discrete system
$\Afr_\Dc$ satisfies the following stability estimate such that for all
$(\ub,p)\in\Hbc_\Tc\times \Hc_\Tc$
\begin{multline}
\sup_{(\vb,r)\in\Hbc_\Tc\times \Hc_\Tc}\frac{\Afr_\Dc
(\ub,p,\vb,r;\lambda,\rho,\tau)}
{\bigl(\norm{\vb}_\Tc^2+\tau \norm{r}_\Tc^2+\rho
\norm{r}_{\Tc,0}^2+\abs{r}_{\Tc,0}^2\bigr)^{1/2}}\\
\ge \Theta^\Afr \bigl(\norm{\ub}_\Tc^2+\tau \norm{p}_\Tc^2+\rho
\norm{p}_{\Tc,0}^2+\abs{p}_{\Tc,0}^2\bigr)^{1/2}
\label{eq5.16}
\end{multline}
The constant $\Theta^\Afr$ depends on the regularity of the grid and the
material parameters, but is bounded independent of $(\rho,\tau,h)\to 0$.
\end{lemma}

\begin{proof}
Recall the non-optimal stability estimate \eqref{eq5.9}. Furthermore,
let now $\wb\in\Hbc_\Dc$ be the function for which the
supremum is achieved in Lemma~\ref{lem5.11}, scaled such that
$\norm{\wb}_\Tc=\abs{p}_{\Tc,0}$. Then for $(\vb,r)=(\ub-\delta\wb,-p)$
\begin{align}
\Afr_\Dc (\ub,p,\vb,r)&=\Afr_\Dc (\ub,p,\ub,-p)+\delta\Afr_\Dc
(-\wb,0,\ub,p)\label{eq5.17}\\
&\ge\Theta^A\norm{\ub}_\Tc^2 + \tau\Theta^C\norm{p}_\Tc^2 +
h^2\Theta^\Delta\abs{p}_\Tc^2
+ \rho\norm{p}_{\Tc,0}^2 + \delta\Theta^B\abs{p}_{\Tc,0}^2\notag\\
&\quad-(\Theta^\Lambda+\delta\Theta^A) \norm{\ub}_\Tc \abs{p}_{\Tc,0}-\delta
h\theta^B\abs{p}_{\Tc,0} \abs{p}_\Tc
\notag
\end{align}
This inequality holds subject to the negative terms being controlled by
the positive terms, which implies the following conditions:
\begin{equation}
4(\Theta^\Lambda+\delta\Theta^A)^2\le
\gamma_1\Theta^A \delta\Theta^B
\quad\text{and}\quad
4(\delta h\theta^B)^2\le\gamma_2 h^2 \Theta^\Delta \delta\Theta^B
\label{eq5.18}
\end{equation}
For the two conditions, we find that by differentiating \eqref{eq5.18}a to
obtain the optimal $\delta$, and determining $\gamma_2$ from \eqref{eq5.18}b
implies the two equalities
\begin{equation}
\delta=\frac{\gamma_1 \Theta^B-8\Theta^\Lambda}{8\Theta^A}
\quad\text{and}\quad
\gamma_2=\frac{4\delta(\theta^B)^2}{\Theta^\Delta \Theta^B}
\label{eq5.19}
\end{equation}
Recall that $\gamma_2=1-\gamma_1$, such that we can simplify equations
\eqref{eq5.19} to obtain
\begin{equation}
\gamma_1=\frac{2\Theta^A \Theta^\Delta \Theta^B
(\theta^B)^{-2}+8\Theta^\Lambda}
{2\Theta^A \Theta^\Delta \Theta^B (\theta^B)^{-2}+\Theta^B}
\label{eq5.20}
\end{equation}
We see that equation~\eqref{eq5.20} implies that $\gamma_1>0$, while
for also $\gamma_1<1$ it must hold that
\begin{equation}
8\Theta^\Lambda<\Theta^B
\label{eq5.21}
\end{equation}
While also by substituting \eqref{eq5.20} into either of equations
\eqref{eq5.19} and \eqref{eq5.18}a we obtain:
\begin{equation}
8\Theta^\Lambda\le \gamma_1 \Theta^B
\label{eq5.22}
\end{equation}
However, substituting \eqref{eq5.20} into equation~\eqref{eq5.18},
we note that \eqref{eq5.22} always holds whenever \eqref{eq5.21}
holds, which thus remains the sole condition in the proof. Subject to
this condition we then have the inequality
\[
\Afr_\Dc (\ub,p,\vb,r)\ge C_1 \bigl(\norm{\ub}_\Tc^2+\tau \norm{p}_\Tc^2
+\rho \norm{p}_{\Tc,0}^2+\abs{p}_{\Tc,0}^2\bigr)
\]
Since we have retained explicitly all dependencies, it is clear
that the stability constant $C_1$ depends on all constants $\Theta$
in \eqref{eq5.17}, but is independent of $\rho$, $\tau$ and $h$, as
asserted. Finally, the lemma follows since
\begin{multline}
\norm{\ub-\delta\wb}_\Tc^2+\tau \norm{p}_\Tc^2+\rho
\norm{p}_{\Tc,0}^2+\abs{p}_{\Tc,0}^2\\
\begin{aligned}[b]
&\le
\norm{\ub}_\Tc^2+\tau \norm{p}_\Tc^2+\rho \norm{p}_{\Tc,0}^2+[1+\delta^2]
\abs{p}_{\Tc,0}^2\\
&\le C_2 (\norm{\ub}_\Tc^2+\tau \norm{p}_\Tc^2+\rho
\norm{p}_{\Tc,0}^2+\abs{p}_{\Tc,0}^2)
\end{aligned}
\label{eq5.23}
\end{multline}
Equation~\eqref{eq5.19} shows that $\delta$ is bounded from above
independently of $\rho$, $\tau$ and $h$, thus the constant $C_2$
is also bounded independent of all coefficients stated in the
lemma, and finally we obtain the lemma with the constant defined as
$\Theta^\Afr=C_1/\sqrt{C_2}$.
\end{proof}

\begin{corollary}\label{cor5.13}
The MPFA/MPSA-FV discretization is robust in the sense of
Definition~\ref{def1.1}, and all eigenvalues of the system $\Afr_\Dc
(\ub,p,\vb,r;\rho,\tau)$ are bounded away from zero.
\end{corollary}

{\makeatletter
\def\@begintheorem#1#2{\par\bgroup{{\itshape #1}\ #2. }\ignorespaces}
\makeatother
\begin{remark}\label{rem5.14}
Lemma~\ref{lem5.12} has a requirement on the grid, which states that
the asymmetry of the coupling terms induced by $\Lambda_\Dc$ must be
sufficiently small, relative to (essentially) the inf-sup constant. Since
$\Theta^\Lambda$ can be locally calculated based on the constants in
Definition~\ref{def5.5}, this condition is locally estimated at the
time of discretization assembly for general grids, and can be verified
\emph{a priori} for regular grids. We will return to this condition
in Section~\ref{sec6.3}. Furthermore, following comment~\ref{com3.4},
we note that for the symmetric discretization on simplex grids, local
conditions~\ref{conA} and \ref{conB} are automatically satisfied,
and $\Theta^\Lambda=0$.
\end{remark}
}

\subsection{Compactness results}\label{sec5.4}

The stability established in Lemma~\ref{lem5.12} will suffice to show
convergence after generalizing a few standard results for the uncoupled
discretizations:

\begin{definition}\label{def5.15}
We consider the following continuous extensions of the cell-average
finite volume gradients for discrete functions $\ub\in\Hbc_\Dc$ (and
equivalently for scalar functions $p\in\Hc_\Dc$):
\begin{equation}
\nabla_\Dc\ub(\xb)=(\widetilde{\nabla}\ub)_K
\quad\text{for }K\in\Tc,\text{ where }\xb\in K.
\label{eq5.24}
\end{equation}
Furthermore, we consider the continuous extension of the consistent
gradient
\begin{equation}
\bnabla_\Dc\ub(\xb)=(\bnabla\ub)_K^s
\quad\text{for $K\in\Tc$, $s\in\Tc$ where $\xb\in K_s$}.
\label{eq5.25}
\end{equation}
\end{definition}

\begin{lemma}\label{lem5.16}
Fix $\tau >0$, and let $\Dc_n$ be a family of regular discretization
triplets (in the sense that mesh parameters remain bounded) such
that $h_n\to 0$, as $n\to\infty$. Furthermore, let the conditions of
Lemma~\ref{lem5.12} hold. Then for all $n$, the solutions $(\ub_n,p_n)$
of equations \eqref{eq4.16}--\eqref{eq4.17} exist and are unique,
there exists $(\widetilde{\ub},\widetilde{p})\in(H^1 (\Omega))^d\times
H^1 (\Omega)$, and up to a subsequence (still denoted by $n$) $\Pi_\Tc
\ub_n\to \widetilde{\ub}$ and $\Pi_\Tc p_n\to \widetilde{p}$ converge in
$(L^q (\Omega))^d\times L^q (\Omega)$, for $q\in[1,2d/(d-2+\epsilon))$ as
$h_n\to 0$. Finally, the cell-average finite volume gradients $\nabla_\Dc
\ub_n$ and $\nabla_\Dc p_n$ converges weakly to $\nabla \widetilde{\ub}$
in $(L^2 (\Omega))^{d^2}$ and $\nabla \widetilde{p}$ in $(L^2 (\Omega))^d$,
respectively.
\end{lemma}

\begin{proof}
The proof follows from the stability of the scheme, Lemma~\ref{lem5.12},
and the compactness arguments detailed in \cite{gas08} and \cite{arn02}.
\end{proof}

\begin{lemma}\label{lem5.17}
Consider the same case as in Lemma~\ref{lem5.16}. Then the consistent
gradients $\bnabla_\Dc\ub_n$ and $\bnabla_\Dc p_n$ converges strongly
to $\nabla\widetilde{\ub}$ in $(L^2 (\Omega))^{d^2}$ and $\nabla
\widetilde{p}$ in $(L^2 (\Omega))^d$, respectively.
\end{lemma}

\begin{proof}
The proof is stated in references \cite{age10} (scalar case) and
\cite{nor15} (vector case).
\end{proof}

{\makeatletter
\def\@begintheorem#1#2{\par\bgroup{{\itshape #1}\ #2. }\ignorespaces}
\makeatother
\begin{remark}\label{rem5.18}
In the limiting case $\tau =0$, the regularity of pressure is reduced,
such that $\widetilde{p}\in L^2 (\Omega)$. Consequently, we do not discuss
convergence of $\nabla_\Dc p_n$ and $\bnabla_\Dc p_n$ in this case.
\end{remark}
}

\subsection{Convergence}\label{sec5.5}

We summarize the results of this section in the main convergence result
for the method.
\begin{theo}\label{the1}
Let $\Dc_n$ be a family of regular discretization triplets
(in the sense that mesh parameters remain bounded) such that
$h_n\to 0$, as $n\to\infty$, and let Conditions~\ref{conA},
\ref{conB} and \ref{conC} hold with constants independent of $n$,
and let $8\Theta^\Lambda<\Theta^B$. Then for $\tau >0$ the limit
$(\widetilde{\ub},\widetilde{p})\in(H^1 (\Omega))^d\times H^1 (\Omega)$
of the discrete mixed variational problem \eqref{eq4.16}--\eqref{eq4.17},
and consequently the MPSA/MPFA discretizations for the Biot equations,
is the unique weak solution of the Biot problem \eqref{eq1.1}.
\end{theo}

\begin{proof}
Lemmas~\ref{lem5.16} and \ref{lem5.17} establish that the
limit $(\widetilde{\ub},\widetilde{p})\in(H^1 (\Omega))^d\times
H^1 (\Omega)$ exists, and the appropriate notion of convergence of
the discrete gradients. It remains to show that the solution pair
$(\widetilde{\ub},\widetilde{p})$ is consistent with the weak form of
problem \eqref{eq1.1}. Uniqueness then follows from the uniqueness
of weak solutions to \eqref{eq1.1}. We show consistency by projecting
continuous functions into the discrete spaces, and note that we suppress
the dependence of $\Pi_\Tc$ and $\Pi_{\FV}$ on $n$:

For all $\ub,\vb\in(C^\infty (\Omega))^d$
\begin{equation}
\lim_{n\to\infty}a_{\Dc_n} (\Pi_{\FV}^{\ub,\ub} \Pi_\Tc \ub,\vb)=\int_\Omega
(\C:\nabla\ub):\nabla\vb\,d\xb
\label{eq5.26}
\end{equation}

For all $\ub\in(C^\infty (\Omega))^d$ and $r\in C^\infty (\Omega)$
\begin{equation}
\lim_{n\to\infty}b_{\Dc_n,1} (\Pi_{\FV}^{\ub,\ub} \Pi_\Tc\ub,r)=-\int_\Omega
r \nabla\cdot\ub \,d\xb
\label{eq5.27}
\end{equation}

For all $\vb\in(C^\infty (\Omega))^d$ and $p\in C^\infty (\Omega)$
\begin{equation}
\lim_{n\to\infty}a_{\Dc_n} (\Pi_{\FV}^{\ub,p} \Pi_\Tc p,\vb)=-\int_\Omega
p \nabla\cdot\vb \,d\xb
\label{eq5.28}
\end{equation}

For all $p,r\in C^\infty (\Omega)$
\begin{equation}
\lim_{n\to\infty}c_{\Dc_n} (\Pi_{\FV}^p \Pi_\Tc p,r)=-\int_\Omega
(k\nabla p)\cdot \nabla r \,d\xb
\label{eq5.29}
\end{equation}

For all $p,r\in C^\infty (\Omega)$
\begin{equation}
\lim_{n\to\infty}b_{\Dc_n,1} (\Pi_{\FV}^{\ub,p} \Pi_\Tc p,r)=0
\label{eq5.30}
\end{equation}

For all $\vb\in(C^\infty (\Omega))^d$
\begin{equation}
\lim_{n\to\infty}\int_\Omega \fb_{\ub}\cdot \Pi_\Tc\vb \,d\xb=\int_\Omega
\fb_{\ub} \cdot\vb \,d\xb
\label{eq5.31}
\end{equation}

For all $r\in(C^\infty (\Omega))^d$
\begin{equation}
\lim_{n\to\infty}\int_\Omega f_p\cdot \Pi_\Tc r \,d\xb=\int_\Omega
f_p\cdot r \,d\xb
\label{eq5.32}
\end{equation}

Since the right-hand sides of equations \eqref{eq5.25}--\eqref{eq5.32} form
the weak form of equations \eqref{eq1.1}, and since $C^\infty$ is dense
in $H^1$, it follows that the limit $(\widetilde{\ub},\widetilde{p})$
is a weak solution to \eqref{eq1.1}.
\end{proof}

\subsection{Fully time-discrete scheme}\label{sec5.6}

From equations \eqref{eq1.1} and \eqref{eq4.15}, we note that the
time-iterative scheme for the solution $(\ub_\Tc^j,p_\Tc^j)$ at time-step
$j$, written in terms of the solution at time-step $j-1$ as as
\begin{multline}
\Afr_\Dc (\ub_\Tc^j,p_\Tc^j,\vb,r;\rho,\tau)=
\Afr_\Dc (\ub_\Tc^{j-1},p_\Tc^{j-1},\vb,r;\rho,0) + \Bfr(\Ob,r)
\\
\text{for all } (\vb,r)\in\Hbc_\Tc\times \Hc_\Tc
\label{eq5.33}
\end{multline}
In terms of linear systems, this is equivalent to the matrix system
\[
\afr_\Dc (\rho,\tau)\cdot
\begin{pmatrix}
\ub_\Tc^j\\[2pt]
p_\Tc^j
\end{pmatrix}
= \afr_\Dc (\rho,0)\cdot
\begin{pmatrix}
\ub_\Tc^{j-1}\\[2pt]p_\Tc^{j-1}
\end{pmatrix}
+ \bfr\cdot
\begin{pmatrix}
\Ob\\[2pt]
r
\end{pmatrix}
\]
Where the matrix $\afr_\Dc (\rho,\tau)$ corresponds to the values of
$\Afr_\Dc (\ub,p,\vb,r;\rho,\tau)$ with $(\ub,p,\vb,r)$ chosen as the
canonical basis. Define any error as
{\emergencystretch5pt
\[
e_\Tc^j=
\begin{pmatrix}
\ub_\Tc^j-\ub^j\\[2pt]
p_\Tc^j-p^j
\end{pmatrix}
\]
}%
Then from Lemma~\ref{lem5.12}, and the Poincar\'e inequality, it follows
that for $\tau >0$, we obtain a stable time-discretization since
\[
\sup_{e_\Tc^{j-1}}\frac{\norm{e_\Tc^j}_0}{\norm{e_\Tc^{j-1}}_0} =
\sup_{e_\Tc^{j-1}}\frac{\norm{\afr_\Dc^{-1}(\rho,\tau)\cdot
\afr_\Dc (\rho,0)\cdot e_\Tc^{j-1}}_0}{\norm{e_\Tc^{j-1}}_0} =
\norm{\afr_\Dc^{-1}(\rho,\tau)\cdot \afr_\Dc(\rho,0)}<1
\]
Thus, the time-discrete problem is also stable, independent of time-step,
and convergent due to the consistency of the first-order truncation error
in backward Euler (implicit) discretization (see e.g.\ \cite{lev02}).

\section{Numerical validation}\label{sec6}

Herein, we restrict our attention towards verifying the claims that
the method is A) robust according to Definition~\ref{def1.1}, B)
is stable as stated in Lemma~\ref{lem5.12}, and C) is convergent
as stated in Theorem~\ref{the1}. We refer to the works on related
methods to give a general impression of this class of methods for
complex grids \cite{nor14a,kla12} and solutions with low regularity
\cite{eig05}. Together, the above references give numerical evidence of
the suitability of the individual discretizations $A_\Dc$ and $C_\Dc$.

\subsection{Problem formulation}\label{sec6.1}

As a test problem, we consider the unit square in two spatial
dimensions. We chose as an analytical solution the function
\begin{equation}
\ub(x_1,x_2)=
\begin{pmatrix}
x_1 (1-x_1)\sin(2\pi x_2)\\
\sin(2\pi x_1)\sin(2\pi x_2)
\end{pmatrix}
\label{eq6.1}
\end{equation}
and $p=\ub\cdot\eb_1$. These functions satisfy zero Dirichlet boundary
conditions, and the problem is thus driven by internal source-sink terms
(for the flow equation) and body-forces (for the momentum equation). We
calculate these right-hand side functions analytically according to
equations \eqref{eq1.1}, using unit permeability and Lam\'e parameters,
in the sense that we consider an isotropic material with the properties
\[
\C:\nabla\ub=\nabla\ub+\nabla\ub^T+(\nabla\cdot u)\Ib
\]
This implies a Poisson ratio of $0.25$, which is within the typical
range of natural materials.

We consider three types of grids. Type A is a standard Cartesian grid,
Type B is a simplex grid obtained by bisecting the Cartesian grid,
while Type C is dual-simplex grid obtained by taking the dual grid to
the corresponding Type B grid. Type C grids are related to PEBI grids
and are representative of unstructured grids. For grid types A and C
we use the general method with second-order quadrature for evaluating
the jump terms, while for grid type B we use the simplified symmetric
discretization as discussed in Comment~\ref{com3.4}.

For the grid refinement, we consider $7$ refinement levels, where the
grid vertexes at each refinement level are perturbed randomly by a
factor between $\pm h/2$. Thus we consider so-called ``rough'' grids
\cite{kla06}, without the milder assumption that the grid asymptotically
is only $h^2$ perturbed, the latter of which is typical for methods
derived from mixed finite element using numerical quadrature (see e.g.\
\cite{whe06}). These rough grids are representative of grid regularity
as experienced in geological subsurface applications.

The grids are illustrated in Figure~\ref{fig1}, together with the
structure of the analytical solution \eqref{eq6.1}. For all figures,
the third refinement level is shown.

\begin{figure}[tbh!]
\centering
\includegraphics[width=0.32\textwidth]{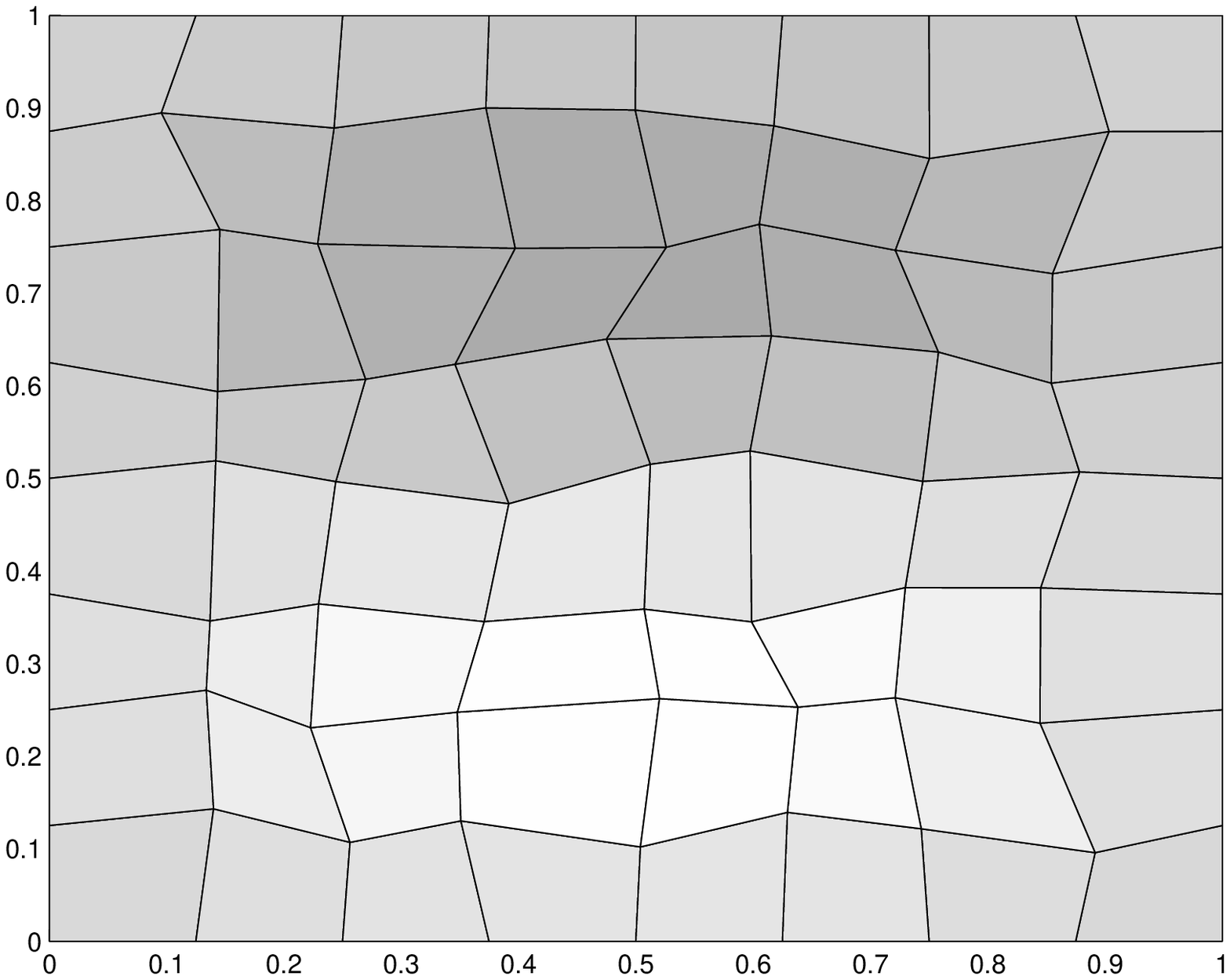}
\includegraphics[width=0.32\textwidth]{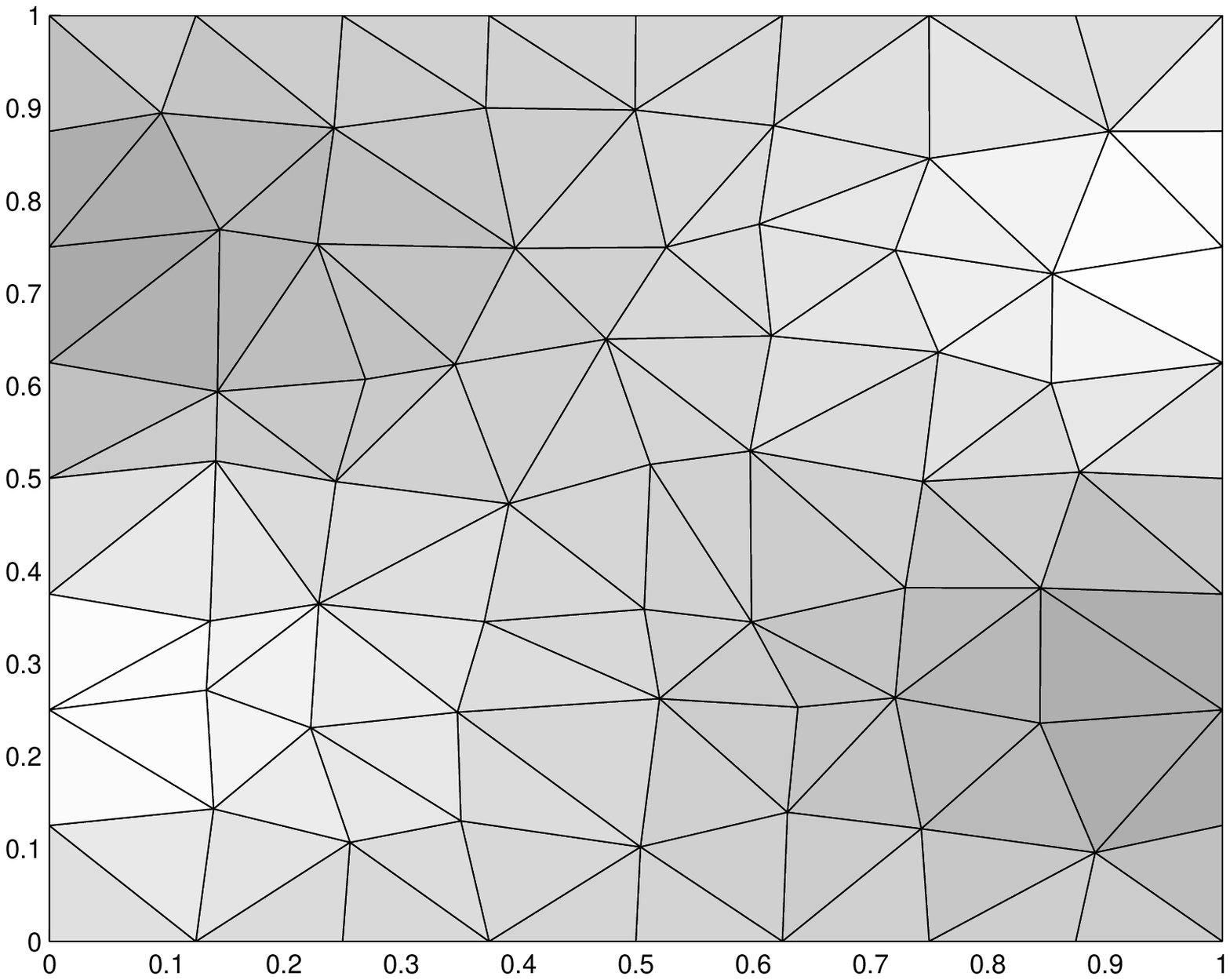}
\includegraphics[width=0.32\textwidth]{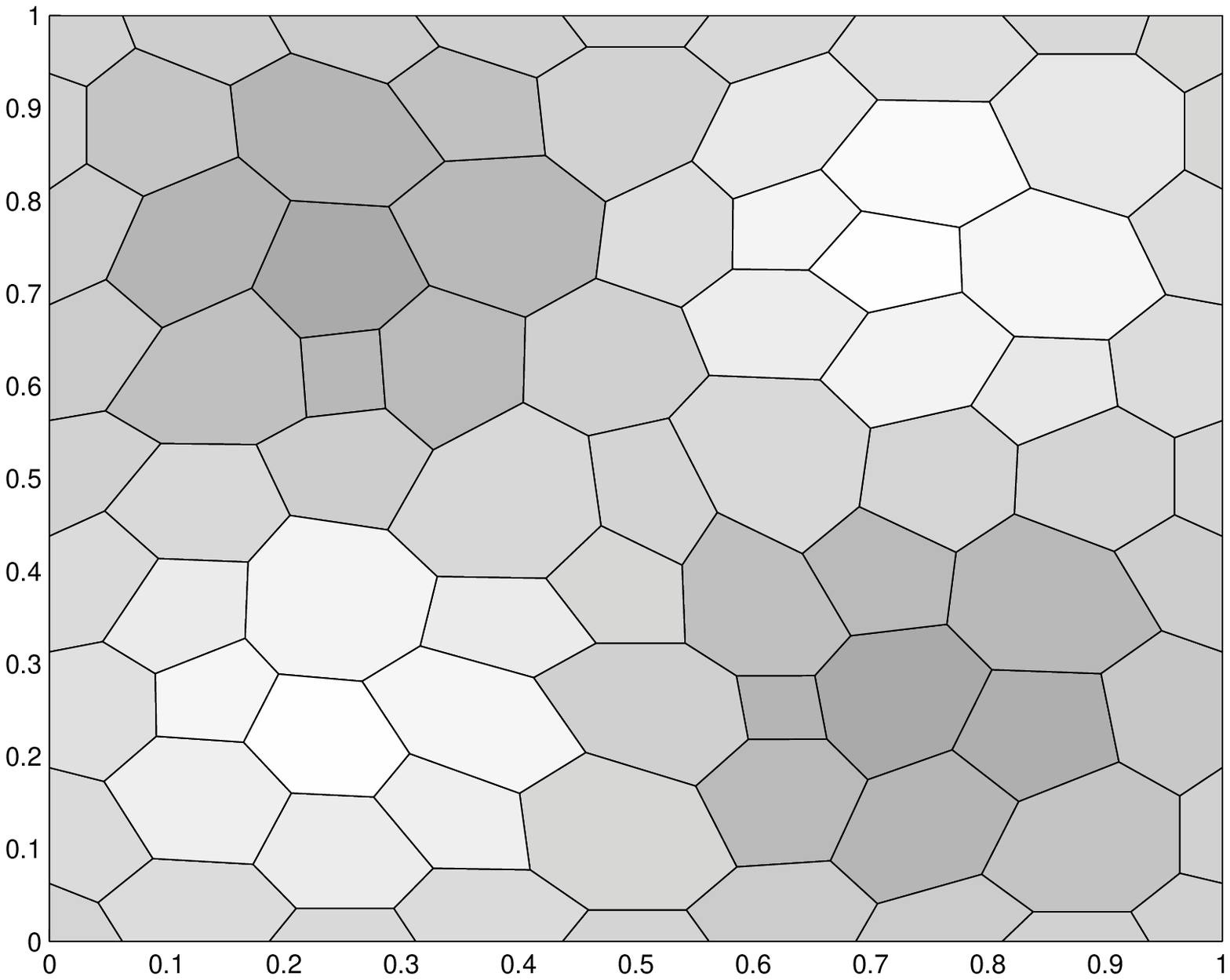}
\caption{From left to right the figures illustrate grid types A
(quadrilaterals), B (triangles), C (unstructured grids). Furthermore,
in grey-scale, the figures indicate the structure of the analytical
solution from Equation~\eqref{eq6.1}, respectively first component of
displacement vector (which is equal to pressure), first component of
pressure gradient, and second component of displacement vector.}
\label{fig1}
\end{figure}

\subsection{Convergence results}\label{sec6.2}

In the reported results, we consider errors using the following $L^2$
type metrics:
\begin{align}
\epsilon_{\ub}&=\frac{\norm{\ub_\Tc-\Pi_\Tc\ub}_{\Tc,0}}{\norm{\Pi_\Tc\ub}_{\Tc,0}}\label{eq6.2}\\
\epsilon_p&=\frac{\norm{p_\Tc-\Pi_\Tc p}_{\Tc,0}}{\norm{\Pi_\Tc
p}_{\Tc,0}}\label{eq6.3}\\
\epsilon_{p,|}&=\frac{\abs{p_\Tc-\Pi_\Tc p}_{\Tc,0}}{\norm{\Pi_\Tc
p}_{\Tc,0}}\label{eq6.4}\\
\epsilon_{\pib}&=\frac{\sum_{\sigma \in\Fc}m_\sigma^2 [\Tb_K^\sigma
-\pib(\xb_\sigma)\cdot\nb_{K,\sigma}]^2}
{\sum_{\sigma \in\Fc}m_\sigma^2
[\pib(\xb_\sigma)\cdot\nb_{K,\sigma}]^2}\label{eq6.5}\\
\epsilon_q&=\frac{\sum_{\sigma \in\Fc}m_\sigma^2
[q_K^\sigma-\qb(\xb_\sigma)\cdot\nb_{K,\sigma}]^2}
{\sum_{\sigma \in\Fc}m_\sigma^2
[\qb(\xb_\sigma)\cdot\nb_{K,\sigma}]^2}\label{eq6.6}
\end{align}
Here the discrete flux and traction, are defined in equations
\eqref{eq3.16} and \eqref{eq3.19}, utilizing that due to equations
\eqref{eq3.17} and \eqref{eq3.20} we can (up to the sign) evaluate
the expression for either $K\in\Tc_\sigma$. In order to show stability
according to Definition~\ref{def1.1}, and verify Lemma~\ref{lem5.12},
we combine the above errors to the \emph{stable error}
\begin{equation}
\epsilon_\Sigma
=\epsilon_{\ub}+\epsilon_{\pib}+(\tau+\rho)\epsilon_p+\tau\epsilon_q+\epsilon_{p,|}
\label{eq6.7}
\end{equation}
In order to illustrate the numerical convergence rate of the primary
variables, we give the \emph{primary error} associated with the primary
variables displacement and pressure as
\begin{equation}
\epsilon_{\ub,p}=\epsilon_{\ub}+\tau\epsilon_p
\label{eq6.8}
\end{equation}
Note that the primary error does not provide control on pressure or flux
in the limit of $\tau\to 0$, we include it in order to highlight better
spatial convergence rates.

For the convergence study, we consider all combinations of
$(\rho,\tau)\in[1,10^{-2},10^{-4},10^{-6}]^2$. All calculations are
performed using Matlab on a standard laptop. The full study contains
$3\times 7\times 5\times 5=525$ calculations, the results of which are
summarized in tables below.

In tables~\ref{tab1} and \ref{tab2} we give the results for the stable
error $\epsilon_\Sigma$ on rough grids. As expected the discretization
is robust, and we obtain (better than) $1^\text{st}$ order convergence independent
of small parameters. Moreover, the error constant is bounded as expected
from Lemma~\ref{lem5.12}.

In tables~\ref{tab3} and \ref{tab4} we give the results for the primary
error $\epsilon_{\ub,p}$ on rough grids. For these results we observe
$2^\text{nd}$ order convergence, again with constant independent of the
small parameters.

We summarize the results of tables~\ref{tab1} through \ref{tab4}
heuristically as:

\begin{statement}\label{sta6.1}
For $h$-perturbed grids the MPFA/MPSA-FV method displays a numerical
convergence following
\begin{equation}
h^{-1} \epsilon_\Sigma +h^{-2} \epsilon_{\ub,p}<C
\label{eq6.9}
\end{equation}
where the constant $C$ does not depend on $\tau$ or $\rho$.
\end{statement}

\subsection{Local conditions}\label{sec6.3}

For all grids in the test suite, we considered the local
conditions~\ref{conA}, \ref{conB}, and \ref{conC}, as well as
the asymmetry condition $8\Theta^\Lambda<\Theta^B$ arising in
Lemma~\ref{lem5.12}. As noted in remark~\ref{rem5.14}, all but
condition~\ref{conC} are immediate on simplex grids. Furthermore, for all
grid of type A and C, the local conditions were satisfied, testifying to
the statement that these conditions are not restrictive in practice. We
emphasize the fact that the local conditions are independent of the
parameters $\rho$ and~$\tau$.

However, this is not to imply that there do not exist grids where
the local conditions are not satisfied. Indeed, similar expressions
as Condition~\ref{conB}, which pertains to the local coercivity of the
discretization of the flow equations, can be violated on severely skewed
parallelograms, as discussed previously \cite{kla06} which can be linked
to a loss of monotonicity of the scheme \cite{nor07}.

\begin{table}[t!]
\caption{Asymptotic convergence rate of stable error $\epsilon_\Sigma$
for grids of types A, B, and C.}\label{tab1}
\centering
\footnotesize
\begin{tabular}{@{}c@{\hskip.8ex}
*{3}{@{\hskip1ex}c@{\hskip1ex}}@{\hskip1ex}
*{3}{@{\hskip1ex}c@{\hskip1ex}}@{\hskip1ex}
*{3}{@{\hskip1ex}c@{\hskip1ex}}@{\hskip1ex}
*{3}{@{\hskip1ex}c@{\hskip1ex}}@{\hskip1ex}
*{3}{@{\hskip1ex}c@{}}}
\toprule
$\epsilon_\Sigma$ & \multicolumn{3}{c}{$\tau =1$} &
\multicolumn{3}{c}{$\tau =10^{-1}$} & \multicolumn{3}{c}{$\tau =10^{-2}$}
& \multicolumn{3}{c}{$\tau =10^{-4}$} & \multicolumn{3}{c}{$\tau
=10^{-6}$}\\
\midrule
Grid & A & B & C & A & B & C & A & B & C & A & B & C & A & B & C\\
\midrule
$\rho =1$ & $1.36$ & $1.36$ & $1.27$ & $1.45$ & $1.51$ & $1.48$ & $1.63$
& $1.68$ & $1.75$ & $2.01$ & $1.90$ & $1.97$ & $1.09$ & $1.14$ & $1.16$\\
$\rho =10^{-1}$ & $1.33$ & $1.32$ & $1.23$ & $1.40$ & $1.46$ & $1.41$
& $1.64$ & $1.66$ & $1.73$ & $2.16$ & $1.89$ & $2.02$ & $1.17$ & $1.25$
& $1.25$\\
$\rho =10^{-2}$ & $1.32$ & $1.32$ & $1.23$ & $1.40$ & $1.45$ & $1.40$
& $1.64$ & $1.66$ & $1.73$ & $2.16$ & $1.88$ & $2.00$ & $1.20$ & $1.29$
& $1.28$\\
$\rho =10^{-4}$ & $1.32$ & $1.32$ & $1.23$ & $1.40$ & $1.45$ & $1.40$
& $1.64$ & $1.66$ & $1.73$ & $2.16$ & $1.88$ & $2.00$ & $1.20$ & $1.29$
& $1.29$\\
$\rho =10^{-6}$ & $1.32$ & $1.32$ & $1.23$ & $1.40$ & $1.45$ & $1.40$
& $1.64$ & $1.66$ & $1.73$ & $2.16$ & $1.88$ & $2.00$ & $1.20$ & $1.29$
& $1.29$\\
\bottomrule
\end{tabular}
\end{table}

\begin{table}[t!]
\caption{Asymptotic stable error $\epsilon_\Sigma$ for grids of types A,
B and C.}\label{tab2}
\centering
\footnotesize
\begin{tabular}{@{}c@{\hskip.8ex}
*{3}{@{\hskip1ex}c@{\hskip1ex}}@{\hskip1ex}
*{3}{@{\hskip1ex}c@{\hskip1ex}}@{\hskip1ex}
*{3}{@{\hskip1ex}c@{\hskip1ex}}@{\hskip1ex}
*{3}{@{\hskip1ex}c@{\hskip1ex}}@{\hskip1ex}
*{3}{@{\hskip1ex}c@{}}}
\toprule
$\epsilon_\Sigma $ & \multicolumn{3}{c}{$\tau =1$} &
\multicolumn{3}{c}{$\tau =10^{-1}$} & \multicolumn{3}{c}{$\tau =10^{-2}$}
& \multicolumn{3}{c}{$\tau =10^{-4}$} & \multicolumn{3}{c}{$\tau
=10^{-6}$}\\
\midrule
Grid & A & B & C & A & B & C & A & B & C & A & B & C & A & B & C\\
\midrule
$\rho =1$ & $.006$ & $.005$ & $.006$ & $.006$ & $.006$ & $.005$ & $.009$
& $.012$ & $.010$ & $.029$ & $.034$ & $.022$ & $.092$ & $.121$ & $.062$\\
$\rho =10^{-1}$ & $.006$ & $.005$ & $.006$ & $.005$ & $.005$ & $.005$
& $.009$ & $.012$ & $.010$ & $.026$ & $.040$ & $.026$ & $.133$ & $.171$
& $.088$\\
$\rho =10^{-2}$ & $.006$ & $.005$ & $.006$ & $.005$ & $.005$ & $.004$
& $.009$ & $.012$ & $.010$ & $.026$ & $.043$ & $.029$ & $.146$ & $.187$
& $.096$\\
$\rho =10^{-4}$ & $.006$ & $.005$ & $.006$ & $.005$ & $.005$ & $.004$
& $.009$ & $.012$ & $.010$ & $.027$ & $.044$ & $.029$ & $.148$ & $.189$
& $.098$\\
$\rho =10^{-6}$ & $.006$ & $.005$ & $.006$ & $.005$ & $.005$ & $.004$
& $.009$ & $.012$ & $.010$ & $.027$ & $.044$ & $.029$ & $.148$ & $.189$
& $.098$\\
\bottomrule
\end{tabular}
\end{table}

\begin{table}[t!]
\caption{Asymptotic convergence rate of primary error $\epsilon_{\ub,p}$
for grids of types A, B and C.}
\label{tab3}
\footnotesize
\centering
\begin{tabular}{@{}c@{\hskip.8ex}
*{3}{@{\hskip1ex}c@{\hskip1ex}}@{\hskip1ex}
*{3}{@{\hskip1ex}c@{\hskip1ex}}@{\hskip1ex}
*{3}{@{\hskip1ex}c@{\hskip1ex}}@{\hskip1ex}
*{3}{@{\hskip1ex}c@{\hskip1ex}}@{\hskip1ex}
*{3}{@{\hskip1ex}c@{}}}
\toprule
$\epsilon_{\ub,p}$ & \multicolumn{3}{c}{$\tau =1$} &
\multicolumn{3}{c}{$\tau =10^{-1}$} & \multicolumn{3}{c}{$\tau =10^{-2}$}
& \multicolumn{3}{c}{$\tau =10^{-4}$} & \multicolumn{3}{c}{$\tau
=10^{-6}$}\\
\midrule
Grid & A & B & C & A & B & C & A & B & C & A & B & C & A & B & C\\
\midrule
$\rho =1$ & $2.00$ & $1.97$ & $1.99$ & $1.98$ & $1.93$ & $1.98$ & $1.99$
& $1.94$ & $1.98$ & $1.99$ & $1.95$ & $1.99$ & $1.99$ & $1.95$ & $1.98$\\
$\rho =10^{-1}$ & $2.00$ & $1.97$ & $1.99$ & $1.98$ & $1.92$ & $1.98$
& $1.99$ & $1.93$ & $1.97$ & $1.99$ & $1.94$ & $1.98$ & $1.98$ & $1.94$
& $1.98$\\
$\rho =10^{-2}$ & $2.00$ & $1.97$ & $1.99$ & $1.98$ & $1.92$ & $1.98$
& $1.99$ & $1.93$ & $1.97$ & $1.98$ & $1.93$ & $1.97$ & $1.98$ & $1.94$
& $1.98$\\
$\rho =10^{-4}$ & $2.00$ & $1.97$ & $1.99$ & $1.98$ & $1.92$ & $1.98$
& $1.99$ & $1.93$ & $1.97$ & $1.98$ & $1.93$ & $1.97$ & $1.98$ & $1.94$
& $1.98$\\
$\rho =10^{-6}$ & $2.00$ & $1.97$ & $1.99$ & $1.98$ & $1.92$ & $1.98$
& $1.99$ & $1.93$ & $1.97$ & $1.98$ & $1.93$ & $1.97$ & $1.98$ & $1.94$
& $1.98$\\
\bottomrule
\end{tabular}
\end{table}

\begin{table}[t!]
\caption{Asymptotic primary error $\epsilon_{\ub,p}$ for grids of types A,
B and C. All numbers are scaled by $10^{-3}$.}
\label{tab4}
\footnotesize
\centering
\begin{tabular}{@{}c@{\hskip.8ex}
*{3}{@{\hskip1ex}c@{\hskip1ex}}@{\hskip1ex}
*{3}{@{\hskip1ex}c@{\hskip1ex}}@{\hskip1ex}
*{3}{@{\hskip1ex}c@{\hskip1ex}}@{\hskip1ex}
*{3}{@{\hskip1ex}c@{\hskip1ex}}@{\hskip1ex}
*{3}{@{\hskip1ex}c@{}}}
\toprule
$\epsilon_{\ub,p}$ & \multicolumn{3}{c}{$\tau =1$} &
\multicolumn{3}{c}{$\tau =10^{-1}$} & \multicolumn{3}{c}{$\tau =10^{-2}$}
& \multicolumn{3}{c}{$\tau =10^{-4}$} & \multicolumn{3}{c}{$\tau
=10^{-6}$}\\
\midrule
Grid & A & B & C & A & B & C & A & B & C & A & B & C & A & B & C\\
\midrule
$\rho =1$ & $1.14$ & $0.84$ & $0.69$ & $0.91$ & $0.72$ & $0.51$ & $0.87$
& $0.67$ & $0.48$ & $0.85$ & $0.64$ & $0.47$ & $0.85$ & $0.64$ & $0.47$\\
$\rho =10^{-1}$ & $1.14$ & $0.84$ & $0.69$ & $0.92$ & $0.74$ & $0.52$
& $0.89$ & $0.71$ & $0.52$ & $0.88$ & $0.71$ & $0.58$ & $0.88$ & $0.71$
& $0.58$\\
$\rho =10^{-2}$ & $1.14$ & $0.84$ & $0.69$ & $0.92$ & $0.74$ & $0.52$
& $0.89$ & $0.72$ & $0.53$ & $0.90$ & $0.75$ & $0.63$ & $0.90$ & $0.74$
& $0.63$\\
$\rho =10^{-4}$ & $1.14$ & $0.84$ & $0.69$ & $0.92$ & $0.74$ & $0.52$
& $0.89$ & $0.72$ & $0.53$ & $0.90$ & $0.75$ & $0.64$ & $0.90$ & $0.75$
& $0.64$\\
$\rho =10^{-6}$ & $1.14$ & $0.84$ & $0.69$ & $0.92$ & $0.74$ & $0.52$
& $0.89$ & $0.72$ & $0.53$ & $0.90$ & $0.75$ & $0.64$ & $0.90$ & $0.75$
& $0.64$\\
\bottomrule
\end{tabular}
\end{table}

\section{Conclusion}\label{sec7}

We provide a new finite volume method for Biot's equations. The method
is distinguished by the following properties:
\begin{itemize}
\item The variables are cell-centered for both displacement and pressure,
allowing for sparse linear systems and efficient data structure. This
is in contrast to industry-standard discretizations using staggered grids.
\item The discretization is valid for general grids in 2D and 3D,
with only local constraints on grid and parameters. Numerical examples
indicate that these constraints are mild.
\item The discretization is naturally stable in the sense of
Definition~\ref{def1.1}.
\end{itemize}

For the proposed discretization, we provide stability analysis and
convergence proof utilizing the framework of hybrid finite volume
methods. The analysis exploits techniques from variational multiscale
methods and stabilized mixed finite element methods, and does not rely
on assumptions of asymptotic grid regularity.

Finally, we provide numerical evidence justifying the stability
analysis. The numerical examples indicate that for rough grids
($h$-perturbed), a general $2^\text{nd}$-order convergence of the
scheme in terms of primary variables (displacement and pressure). The
convergence of secondary variables (traction and normal fluxes) is
in general $1^\text{st}$ order. In the limit where $\tau \to 0$, the
convergence rate of pressure is reduced to $1^\text{st}$ order, while
the convergence of fluxes is lost. This is consistent with the expected
regularity of the solution.

\subsection*{Acknowledgements}
The author is currently associated with the Norwegian Academy of Science
and Letters through VISTA -- a basic research program funded by Statoil.

\bibliographystyle{siam}
\bibliography{biblio}

\end{document}